\documentclass[reqno,a4paper]{amsart}
\usepackage[latin1]{inputenc}
\usepackage[english]{babel}
\usepackage{eucal,amsfonts,amssymb,amsmath,amsthm,epsfig,mathrsfs}
\usepackage{dsfont,cancel,soul}
\usepackage{color}
\usepackage{amscd,amsxtra}
\usepackage{enumerate}
\usepackage{latexsym}
\usepackage{bm}

\allowdisplaybreaks

 \newcounter{ipotesi}

\Alph{ipotesi}
 \makeatletter \@addtoreset{equation}{section}
 
 \makeatother \makeatletter

\newtheorem{thm}{Theorem}[section]
\newtheorem{hyp}[thm]{Hypotheses}{\rm}
\newtheorem{hyp0}[thm]{Hypothesis}{\rm}

\newtheorem{prop}[thm]{Proposition}

\newtheorem{rmk}[thm]{Remark}{\rm}
\newtheorem{example}[thm]{Example}

\newcounter{parentenv}

\newcommand{\R}{{\mathbb R}}

\newcommand{\N}{{\mathbb N}}

\newcommand{\Rd}{\mathbb R^d}
\newcommand{\Rn}{\mathbb R^n}
\newcommand{\Rm}{\mathbb R^m}
\newcommand{\Cm}{\mathbb C^m}

\newcommand{\g}{{\bm g}}

\newcommand{\f}{{\bm f}}

\newcommand{\uu}{{\bm u}}
\newcommand{\A}{\bm{\mathcal A}}

\newcommand{\vv}{{\bm v}}
\newcommand{\ww}{{\bm w}}

\renewcommand{\tilde}[1]{\widetilde{#1}}

\begin{document}

\title[Strongly coupled Schr\"odinger operators]{Strongly coupled Schr\"odinger operators in $L^p(\Rd;\Cm)$}
\author[L. Angiuli, L. Lorenzi and E.M. Mangino ]{Luciana Angiuli, Luca Lorenzi, Elisabetta M. Mangino}
\address{L.A. \& E.M.M.:  Dipartimento di Matematica e Fisica ``Ennio De Giorgi'', Universit\`a del Salento, Via per Arnesano, I-73100 LECCE, Italy}
\address{L.L.: Dipartimento di Scienze Matematiche, Fisiche e Informatiche, Plesso di Mate\-matica, Universit\`a degli Studi di Parma, Parco Area delle Scienze 53/A, I-43124 PARMA, Italy}
\email{luciana.angiuli@unisalento.it}
\email{luca.lorenzi@unipr.it}
\email{elisabetta.mangino@unisalento.it}
\thanks{This article is based upon work from COST Action CA18232 MAT-DYN-NET, supported by COST (European Cooperation in Science and Technology). The authors are also members of G.N.A.M.P.A. of the Italian Istituto Nazionale di Alta Matematica (INdAM)}

\begin{abstract}
We consider systems of elliptic equations, possibly coupled up to the second-order, on the $L^p(\Rd;\Cm)$-scale. Under suitable assumptions we prove that the minimal realization in $L^p(\Rd;\Cm)$ generates a strongly continuous analytic semigroup. We also prove the consistency of the semigroup on the $L^p$-scale and some spectral results.
\end{abstract}

\maketitle

\section{Introduction}

The aim of this paper is the study of the solvability of the Cauchy problem
\begin{equation}\label{intro1}
\left\{
\begin{array}{lll}
\displaystyle D_t\uu=\sum_{h,k=1}^d Q^{hk}D_{hk} \uu - V\uu, &t\in(0,\infty), &x\in\Rd,\\[4mm]
\uu(0,x)=\f(x), && x\in\Rd,
\end{array}
\right.
\end{equation}
where  $Q^{hk}$ ($h,k=1, \ldots, d$) and $V$ are $m\times m$ matrix-valued functions, $V$ is allowed to be unbounded   and $\f$ is a vector-valued function on $\Rd$ whose components belong to $L^p(\Rd)$.
The theoretical frame in which we face the problem is  semigroup theory, so  our main aim consists in proving generation results in   $L^p(\Rd;\Cm)$ for the vector-valued operator $\A$, defined on smooth vector-valued functions $\uu:\R^d\to\Cm$ by
\begin{equation}\label{operator}
\A\uu=\sum_{h,k=1}^d Q^{hk}D_{hk} \uu-V\uu=\A^0\uu-V\uu.
\end{equation}

The interest in this type of problems is two-fold. On one hand, they appear naturally in several applied  research fields such as the study of Navier-Stokes equation or of Nash equilibria in game theory, time dependent Born-Oppenheimer theory, stochastic differential equations. On the other hand, from a pure mathematical point of view, these problems turn to be quite challenging, since the techniques to tackle them are not a mere translation to the vector-valued case of those applied to the study of  linear elliptic/parabolic equations.   This feature  has been made evident  in several papers about elliptic operators with complex coefficients, which can be clearly interpreted as vector-valued operators with real coefficients. Actually, the system appearing in the first equation of \eqref{intro1} has second-order terms that can be coupled to each other and consequently, as it is proved in \cite{Kre_maz}, a maximum modulus principle fails to hold. This yields, e.g., that  the strategy adopted in \cite{AngLor Pal}, namely to take advantage of the generation results of a semigroup $T(t)$ in the setting of  bounded and continuous functions (see \cite{AALT, AL}) and to extrapolate $T(t)$ to a strongly continuous semigroup in the $L^p$-scale,  although working for a large class of operators coupled up to the first-order, cannot be applied to the operator \eqref{operator}.
 Thus, in order to prove generation results in $L^p(\Rd;\Cm)$ for  a suitable realization of the operator $\A$,  providing also an explicit description of its domain, direct methods look more appropriate. For example, one could apply  classical generation theorems, as  the Lumer-Phillips theorem. This approach worked quite satisfactorily in the case of operators coupled up to the first-order, even with unbounded diffusion coefficients (see e.g., \cite{AngLorMan1,ALMR, ALMR1}), allowing to prove generation results in $L^p(\Rd;\mathbb C^m)$ for any $p\in (1,\infty)$.  But, when dealing with operators coupled at the second-order,   $L^p$-dissipativity, thoroughly investigated in  \cite[Section 4.3]{CM}, is a critical point.  Actually,  one cannot expect dissipativity to hold true in $L^p(\Rd;\mathbb C^m)$ for any $1\leq p\leq \infty$, unless $Q^{hk}=q_{hk} I$, where $q_{hk}$ is a scalar function. By the way,
in \cite{AngLorMan}, strongly coupled Schro\"edinger type operators in divergence form have been considered and, under suitable assumptions on the (possibly unbounded)  diffusion matrices $Q^{hk}$,  regular dissipativity and generation results of analytic semigroups in $L^p(\Rd;\Cm)$ have been proved, provided that $p$ satisfies the condition
$
\left\vert \frac 1 p -\frac 1 2\right\vert\leq K,
$
where $K$ is a constant depending on the coefficients of the diffusion part. To the best of our knowledge, the latter are the  only  results in literature about   vector-valued elliptic operators coupled at the second-order with unbounded coefficients  in the $L^p$-setting.

In this  paper we are interested in getting  generation results in $L^p(\R^d;\Cm)$, which hold true for any $p\in (1,\infty)$, and we will pursue this goal by applying a Dore-Venni type result for the sum of non-commuting operators due to Monniaux and Pruss (see  Theorem \ref{M_P}).

This approach has been already used for weakly coupled operators in the $L^p$-setting in \cite{HLPRS,KLMR}. Here, assuming that the diffusion coefficients are bounded and uniformly continuous and the potential term $V$ satisfies two alternative set of assumptions, see Hypotheses \ref{H1-1} and \ref{hyp_2}, we prove that the realization $A_p$ of $\A$ in $L^p(\Rd;\Cm)$, endowed with its minimal domain $D(A_p)=\{\uu\in W^{2,p}(\Rd;\Cm):V\uu\in L^p(\Rd;\Cm)\}$ generates a strongly continuous and analytic semigroup in $L^p(\Rd;\Cm)$
for any $p\in (1,\infty)$. We point out that in the weakly coupled case (i.e., when $Q^{hk}=q_{hk}I$ for every $h,k=1,\ldots,d$ and $q_{hk}$ is a scalar function), our assumptions on the potential $V$ are less restrictive than those considered in \cite{KLMR} and allow potential matrices $V$ with entries that can grow more than quadratically, which is the best case considered in \cite{KLMR}. The consistency of the semigroups in the $L^p$-scale is proved, as well as some inclusions of the spectra of the operator $A_p$ as $p$ varies in $ (1,\infty)$. Concerning this last topic, as it is well-known already in the scalar case, the spectrum of the realization of an elliptic operator in the $L^p$-space related to the Lebesgue measure can depend on $p$. In \cite{H_V_86}, the authors provide conditions on the potential term to ensure the $p$ independence of the spectrum of the Schr\"odinger operator $\Delta-v$ in $L^p(\Rd)$. In \cite{H_V_87}, allowing rather singular potentials, the same authors provide only some inclusions of the spectra of the same operator. In an abstract setting, in \cite{HS} certain commutator conditions on the semigroup and on the resolvent of its generator are provided to ensure the $p$ independence of the spectrum. Applications include the case of Petrovskij correct systems with H\"older-continuous coefficients, Schr\"odinger operators, and certain elliptic operators in divergence form with real or complex coefficients.  More recently, the $L^p$-spectrum of the one dimensional Schr\"odinger operator $-\frac{d}{dx^2}-x^2$ with domain $C_0^\infty(\R)$ has been completely characterized in \cite{Fin_isi}, showing that it consists of the strip $\{\lambda\in\mathbb C: |{\rm Im}\lambda|\le \frac{2}{p}-1\}$, if $p \le 2$, and of the strip $\{\lambda\in\mathbb C: |{\rm Im}\lambda|\le 1- \frac{2}{p}$\}, if $p>2$.
In a more general setting, the question on the $p$-dependence of the spectrum of operators in $L^p$-spaces has been extended in various directions. For several different aspects see e.g. \cite{arendt, davies, dav_sim_tay,ji_web,sturm} and the references therein.
Further, even if the whole spectrum does not depend on $p$, this could be not the case for some of its proper subsets. For instance the point spectrum of the Laplacian operator in $L^p(\Rd)$  is empty if $p \le 2d/(d-1)$ and is equal to $(-\infty,0)$ otherwise (see \cite{ta}).

In our case, we show that the point spectrum of $A_p$ is a subset of the point spectrum of
$A_q$ and, conversely, every eigenvalue of $A_q$ belongs to the approximate spectrum of $A_p$ for any $1<p<q<\infty$. Clearly, if the spectrum of $A_p$ consists of eigenvalues only, by the previous two inclusions we deduce that $\sigma(A_p)=\sigma(A_q)$ for any $p,q \in (1,\infty)$. This is what happens, for instance, if $|V(x)\eta|\ge \psi(x)|\eta|$ for any $x \in \Rd$, $\eta \in \Rm$ and some function $\psi:\Rd \to \R$ diverging to $\infty$ as $|x|$ tends to $\infty$. Some further results are obtained on the spectrum of $A_p$ when the operator $\A$ is in divergence form and the potential matrix $V$ is symmetric.

\medskip

\paragraph{\bf Notation.}
If $d, m\in\N$ and  $\mathbb K=\mathbb R$ or $\mathbb K=\mathbb C$,  we denote  by $(\cdot, \cdot)$ and $|\cdot|$, respectively, the Euclidean inner product and the Euclidean norm in $\mathbb K^m$.
Vector-valued functions will be displayed in bold style and, given a ${\mathbb K}^m$-valued function  $\uu$,   its $k$-th component will be denoted by  $u_k$. For every $p\in [1,\infty)$,  $L^p(\Rd;\mathbb K^m)$  is the classical vector-valued Lebesgue space endowed with the norm
$\|\f\|_p=(\int_{\Rd} |\f(x)|^pdx)^{1/p}$,
while, for $k\in\mathbb N$, $W^{k,p}(\Rd, \mathbb K^m)$ is the classical vector-valued Sobolev space, i.e., the space of all functions $\uu\in L^p(\Rd;\mathbb K^m)$ whose components have distributional derivatives up to the order $k$, which belong to $L^p(\R^d;\mathbb K)$, endowed with  the usual norm
$\|\cdot\|_{k,p}$. When $\mathbb K=\R$ and $m=1$, we simply write $L^p(\Rd)$ and $W^{k,p}(\Rd)$.
If $X(\Rd;\mathbb{K}^m)$ is one of the functional spaces above,  $X_{\rm{loc}}(\Rd;\mathbb{K}^m)$ stands for  the set of functions which belong to $X(\mathcal{K};\mathbb{K}^m)$ for every compact set $\mathcal{K}\subset \Rd$.
By $C^{\infty}_c(\Rd;\mathbb K^m)$, we denote the set all the vector-valued functions which have compact support in $\Rd$ and are infinitely many times differentiable. Similarly, for every $k\in\N$,
$C^k_c(\Rd;\mathbb K^m)$ denotes the set of all the compactly supported functions $\uu:\Rd\to\mathbb K^m$ which are continuously differentiable on $\Rd$, up to the $k$-th order, while $C^k_b(\Rd;\mathbb K^m)$ is the space  of continuously differentiable functions on $\Rd$, up to the $k$-th order, which are  bounded  together with their derivatives. When $\mathbb K=\R$ and $m=1$, we simply write $C^{\infty}_c(\Rd)$, $C^k_c(\Rd)$, $C^k_b(\Rd)$.

If $X,Y$ are Banach spaces and $T:X\rightarrow Y$ is a bounded  linear operator, $\|T\|_{\mathcal L(X,Y)}$ denotes its operator norm.
If $T:D(T)\subseteq X\rightarrow Y$ is a linear closed operator, as usual, we denote by $\sigma(T)$, $P\sigma(T)$, $A\sigma(T)$ respectively its spectrum, its point spectrum and its approximate spectrum.

For any matrix $M\in \mathbb{C}^{n^2}$, we set $|M|=\|M\|_{\mathcal L(\mathbb C^n,\mathbb C^n)}$,
Finally,  $B(r)$ denotes the open ball in $\Rd$, centered at zero and with radius $r$.

\section{Preliminaries}For sake of clearness, we start by recalling some well-known notions whose definition sometimes varies depending on the point of view of the research lines. Along this paper, following the definition of Monniaux and Pr\"uss, we say that
a densely defined operator $A:D(A)\subseteq X\rightarrow X$, where $X$ is a complex Banach space,  is  \emph{sectorial} if  it is injective, with dense range, $(0,\infty)\subseteq \rho(-A)$ and $M_0=\sup_{r>0}r\|(r+A)^{-1}\|<\infty$. Moreover, $A$ is called \emph{quasi-sectorial} if $\sigma+A$ is sectorial for some $\sigma > 0$.

According to \cite{AHS}, the operator $A$ is said to be of \emph{positive type}, if it is sectorial and $0\in \rho(-A)$.

It is well known that if $A$ is sectorial, then there exists an angle $\varphi\in (0,\pi]$ such that the sector
\begin{eqnarray*}
\Sigma_{\varphi}=\{ z\in\mathbb C\, \mid\, z\not=0,\ |{\rm arg}(z)|< \varphi \}
\end{eqnarray*}
is contained in
$\rho(-A)$
and
\begin{eqnarray*}
M_\varphi=\sup_{\lambda \in \Sigma_\varphi}\|\lambda (\lambda+ A)^{-1}\|_{{\mathcal L}(X)} < \infty.
\end{eqnarray*}

Consequently, $\sigma(-A)\subseteq - \overline{\Sigma_{\pi-\varphi}}$ and, for this reason,  the spectral angle of $A$ is defined by
\begin{eqnarray*}
\varphi_A= \inf\{ \varphi\in(0,\pi)\, \mid\,\Sigma_{\pi-\varphi} \subseteq \rho(-A),\  M_{\pi-\varphi}<\infty\}.
\end{eqnarray*}

It is well known, see e.g., \cite[Theorem II.4.6]{engnagel}, that, if $A$ is a sectorial operator with spectral angle $\varphi_A<\frac \pi 2$, then $-A$ generates a bounded analytic semigroup.

For an operator of positive type, it is possible to prove the existence of an angle $\theta\in (0,\pi)$  and a positive constant $K\geq1$ such that
\begin{eqnarray}
\label{positive}
\overline{\Sigma_{\theta}}=\{ z\in\mathbb C : |{\rm arg}(z)|\leq \theta\} \subseteq \rho(-A),\qquad\;\,\sup_{\lambda \in \overline{\Sigma_\theta}}\|(1+|\lambda|) (\lambda+ A)^{-1}\|_{\mathcal{L}(X)} \leq K.
\end{eqnarray}

If we need to pay particular attention to the angle $\theta$ and the constant $K$ appearing in \eqref{positive}, we  will say that an operator belongs to the class $\mathcal P(K,\theta)$ when the conditions in \eqref{positive} are satisfied.

Consequently, if  $A\in \mathcal P (K , \theta)$, then  $A$  is sectorial with spectral angle  $\varphi_A\leq\pi- \theta$.

Before going on, we recall that for any operator $A$ of positive type, the operator $A^\alpha$,  (see \cite[Chapter 4]{LunInt} for a definition) is well defined for every $\alpha\in (0,\infty)$ and enjoy good properties as the next theorem states.

\begin{thm}\label{power_thm} Let $A$ belong to the class $\mathcal{P}(K, \theta)$ for some $K\ge 1$ and $\theta\in (0,\pi)$.
Then, the following properties are satisfied.
\begin{enumerate}[\rm(i)]
    \item
The operator $A^\alpha$ is well defined for every $\alpha>0$;
\item
each operator $A^{\alpha}$ commutes with $(\lambda+A)^{-1}$ on $D(A^\alpha)$ for every $\lambda$ in the resolvent set of the operator $-A$;
 \item
 for any $\alpha >0$ there exists a positive constant $C=C(\alpha, K, \theta)$ such that
 \begin{equation}\label{crucial}
 \|A^\alpha (\lambda+A)^{-1}\|_{{\mathcal L}(X)}\leq \frac{C}{(1+|\lambda|)^{1-\alpha}},\qquad\;\, \lambda \in\overline\Sigma_{\theta}.
 \end{equation}
 \end{enumerate}

\end{thm}

\begin{proof}
Property (i)  follows from \cite[Chapter 3]{LunInt}.
Concerning property (ii), it suffices to prove it only for  $\alpha\in (0,1)$. In this case, thanks to the Balakrishnan formula (see e.g., \cite[formula (4.1.7)]{LunInt}), we can write
\begin{equation}\label{Balak}
A^\alpha x=\frac{1}{\Gamma(\alpha)\Gamma(1-\alpha)} A\int_0^{\infty} \xi^{\alpha-1}(\xi+A)^{-1}xd\xi, \qquad x\in D(A^\alpha).
\end{equation}
By assumptions, $\overline{\Sigma_\theta}\subseteq \rho(-A)$ and
\begin{equation*}
\|(1+|\lambda|)(\lambda+A)^{-1}\|_{{\mathcal L}(X)}\leq K,\qquad\;\,\lambda\in \overline{\Sigma_{\theta}},
\end{equation*}

Thus, for every $x\in D(A^\alpha)$ and $\lambda \in \overline{\Sigma_{\theta}}$ it follows that
\begin{align*}
(\lambda+A)^{-1}A^\alpha x&=\frac{1}{\Gamma(\alpha)\Gamma(1-\alpha)} (\lambda+A)^{-1} A\int_0^{\infty} \xi^{\alpha-1}(\xi+A)^{-1}xd\xi\\
&=\frac{1}{\Gamma(\alpha)\Gamma(1-\alpha)} A(\lambda+A)^{-1} \int_0^{\infty} \xi^{\alpha-1}(\xi+A)^{-1}xd\xi\\
&=\frac{1}{\Gamma(\alpha)\Gamma(1-\alpha)}  A\int_0^{\infty} \xi^{\alpha-1}(\lambda+A)^{-1}(\xi+A)^{-1}xd\xi\\
&=\frac{1}{\Gamma(\alpha)\Gamma(1-\alpha)}  A\int_0^{\infty} \xi^{\alpha-1}(\xi+A)^{-1}(\lambda+A)^{-1}xd\xi\\
&=A^\alpha (\lambda+A)^{-1}x.
\end{align*}

Finally, to prove estimate \eqref{crucial} in (iii), we observe that, by \eqref{Balak} and by applying \cite[Proposition 3.1.1]{LunInt}, we get that, for every $x\in (X, D(A))_{\alpha,1}$,
\begin{align}\label{pre-1}
\|A^\alpha x\|&\leq \frac{1}{\Gamma(\alpha)\Gamma(1-\alpha)} \int_0^{\infty} \xi^{\alpha-1}\|A(\xi+A)^{-1}x\|d\xi \notag\\
&=\frac{1}{\Gamma(\alpha)\Gamma(1-\alpha)}
\|\xi\mapsto \xi^\alpha A(\xi+A)^{-1}x\|_{L^1((0,\infty), \frac{d\xi}{\xi})}\notag\\
    &\leq \frac{1}{\Gamma(\alpha)\Gamma(1-\alpha)} \|x\|_{(X, D(A))_{\alpha,1}}.
\end{align}

Consequently,  it follows that
\begin{align*}
\|A(\lambda+A)^{-1}x\|&=\|-x+\lambda (\lambda+ A)^{-1}x\|
\leq \|x\| + |\lambda|\|(\lambda+A)^{-1}x\| \leq (K+1)\|x\|.
\end{align*}
Hence,
\begin{align*}
\|(\lambda+A)^{-1}x\|_{D(A)}\leq \frac{K}{|\lambda|+1}\|x\|+ (K+1)\|x\|\leq (1+2K)\|x\|.
\end{align*}

By \cite[Proposition 1.1.6]{LunInt} (applied with $X_1=X_2=Y_1=X$, $Y_2=D(A)$, $\theta=\alpha$, $p=1$), we conclude that
\begin{align*}
\|(\lambda+A)^{-1}\|_{\mathcal{L}((X, D(A))_{\alpha, 1})}\le & \|(\lambda+A)^{-1}\|^{1-\alpha}_{\mathcal L(X)} \|(\lambda+A)^{-1}\|^{\alpha}_{\mathcal L(X, D(A))}\\
\le &\bigg(\frac{K}{1+|\lambda|}\bigg)^{1-\alpha} (1+2K)^{\alpha}=:\frac{C'(\alpha)}{(1+|\lambda|)^{1-\alpha}}.
\end{align*}
Therefore, taking \eqref{pre-1} into account, we deduce
\begin{align*}
\|A^\alpha (\lambda+A)^{-1}x\| &\leq \frac{1}{\Gamma(\alpha) \Gamma (1-\alpha)}\|(\lambda+A)^{-1}x\|_{(X, D(A))_{\alpha,1}}\\
&\leq  \frac{1}{\Gamma(\alpha) \Gamma (1-\alpha)} \frac{C'(\alpha)}{(1+|\lambda|)^{1-\alpha}},
\end{align*}
whence the claim.
\end{proof}

To go further, we recall that a sectorial operator $A$ is said to admit \emph{bounded imaginary powers} if the purely imaginary powers $A^{is}$ of $A$ are uniformly bounded for $s \in [-1,1]$. In this case, $(A^{is})_{s\in\R}$ forms a strongly continuous $C_0$-group of bounded linear operators. The class of such operators is denoted by $BIP(X)$. The type $\theta_A$ of the $C_0$-group $(A^{is})_{}$ is called the power angle of $A$, i.e.
\begin{equation*}
\theta_A:=\inf \big\{\omega \ge 0 \,\,|\,\, \exists M>0\,\, \|A^{is}\|_{\mathcal{L}(X)}\le Me^{\omega|s|}, \quad s \in \R\big\}.
\end{equation*}
It is well-known that $\theta_A \ge \varphi_A$ (see \cite[Theorem 2]{Pru_Soh}).

The key instrument  for the generation results we mean to prove is a well-known result due to Monniaux and Pr\"uss about the sum of two non-commuting operators.
Since we will need to pay attention to the interrelations between the constants appearing in the statements, the spectral and power angles and the quantitative properties of the coefficients of the differential operator, we state the result by Monniaux and Pr\"uss in a slightly different way, limiting ourselves to the $L^p$-setting and providing the argument that permits to justify the clarifications in the statement.

\begin{thm}\label{M_P}
Let $1<p<\infty$ and let  $L_1,L_2$ be two closed, linear, densely defined operators in $L^p(\R^d; \mathbb{C}^m)$, which are sectorial and admit bounded imaginary  powers. Assume that $L_1$ is invertible and that  $\theta_{L_1} + \theta_{L_2} < \pi$. Further, assume that there exist angles $\gamma_1>\theta_{L_1}$ and $\gamma_2>\theta_{L_2}$ such that $\gamma_1+\gamma_2<\pi$ and
\begin{equation}
\label{commutator_0}
 \|L_1(\lambda+L_1)^{-1}[L_1^{-1}(\nu+L_2)^{-1}-(\nu+L_2)^{-1}L_1^{-1}]\|_{\mathcal{L}(L^p(\R^d; \mathbb{C}^m))}\le \frac{c}{(1+|\lambda|^{1-\beta})|\nu|^{1+\delta}}
\end{equation}
for some constants $c>0$, $0\le \beta<\delta<1$ and any $\lambda \in \Sigma_{\pi -\gamma_1}$ and $\nu \in \Sigma_{\pi-\gamma_2}$.
Then, $L_1 + L_2$, with domain $D(L_1+L_2):=D(L_1) \cap D(L_2)$, is closed and there exists a constant $\sigma_0=\sigma_0(\gamma_2)\ge 0$ such that
 $\sigma_0 + L_1+L_2$ is sectorial. Finally,  $\varphi_{\sigma_0+L_1+L_2}\le\max
 \{\gamma_1,\gamma_2\}$.
\end{thm}

\begin{proof}
By \cite[Corollary 1]{Mon_Pru}, there exists a constant $c_0>0$ such that, if $c<c_0$, then $L_1+L_2$ is sectorial and $\varphi_{L_1+L_2}\leq \max
 \{\gamma_1,\gamma_2\}$.
 Replacing the operator $L_2$ with $L_2+\sigma_0$, with a positive constant $\sigma_0$ and replacing $\delta$ with a smaller $\delta'\in (\beta, \delta)$,  the right-hand side of \eqref{commutator_0} can be estimated by
\begin{eqnarray*}
\frac{c}{(1+|\lambda|^{1-\beta})|\nu+\sigma_0|^{1+\delta}}\le\frac{c'}{(1+|\lambda|^{1-\beta})|\nu+\sigma_0|^{1+\delta'}}
\end{eqnarray*}
for any $\lambda \in \Sigma_{\pi-\gamma_1}$ and $\nu \in \Sigma_{\pi-\gamma_2}$
where
\begin{eqnarray*}
c'=\frac{c}{\sigma_0^{\delta-\delta'}(1-\cos^2 (\phi_0))^{\frac{\delta-\delta'}{2}}}
\end{eqnarray*}
and $\phi_0=\min\{\gamma_2,\frac{\pi}{2}\}$. Thus,
the constant $c'$ can be made smaller than $c_0$, choosing  $\sigma_0$ large enough, so that Theorem 1 in \cite{Mon_Pru} applies to the pair $(L_1,L_2+\sigma_0)$. In particular, from Corollary 1 in \cite{Mon_Pru} the estimate on the spectral angle of $\sigma_0+L_1+L_2$ follows.
\end{proof}

\section{Key tools}

In order to apply Theorem \ref{M_P} we start collecting some known facts about the  diffusion part of the operator $\A^0$, i.e., the operator defined by
\begin{eqnarray*}
(\A^0\uu)(x)= \sum_{h,k=1}^dQ^{hk}(x)D_{hk}\uu(x),\qquad\;\, x \in \R^d,
\end{eqnarray*}
on smooth functions $\uu:\Rd\to\Cm$, where, for any $x \in \R^d$, $Q^{hk}(x)=(q^{hk}_{ij}(x))_{i,j=1}^m$ are real valued $m\times m$ matrices.

Given $M>0$ and $\theta_0\in [0,\pi]$,  $-\A^0$ is said to be uniformly $(M, \theta_0)$-elliptic if
\begin{equation*}
\max_{h,k=1,\ldots,d}\|Q^{hk}\|_{\infty}\le M
\end{equation*}
and, for any $x\in\Rd$, $\xi\in \Rd$ with $|\xi|=1$, it holds that
\begin{equation*}
\sigma(A_0(x, \xi))\subseteq \overline{\Sigma}_{\theta_0}\setminus\{0\},\quad\;\,\quad\,\,|A_0(x, \xi)^{-1}|\leq M,
\end{equation*}
where $A_0(x,\xi)$ denotes the principal symbol of $\A^0$, i.e., the matrix
\begin{align*}
A_0(x, \xi)=\sum_{h,k=1}^d{Q^{hk}(x)\xi_h\xi_k}, \qquad\;\,  x\in\Rd,\;\,\xi=(\xi_1,\ldots,\xi_d)\in \Rd.
\end{align*}

For every $p\in (1,\infty)$, we now define by $A_p^0$ the realization of the operator $\A^0$ in $L^p(\Rd;\Cm)$, with domain $D(A_p^0)=W^{2,p}(\Rd;\Cm)$, i.e., the operator $A_p^0:W^{2,p}(\Rd;\Cm)\to L^p(\Rd;\Cm)$, defined by $A_p^0\uu=\A^0\uu$ for every $\uu\in W^{2,p}(\Rd;\Cm)$.

Let us prove a crucial result for our analysis.

\begin{thm}\label{thm_AHS}
Fix $p \in (1,\infty)$. Assume that the  matrices $Q^{hk}$ belong to $BUC(\Rd; \R^{2m})$ and there exist a positive constant $M$ and $\theta_0 \in [0,\pi]$ such that $-\A^0$ is uniformly $(M, \theta_0)$-elliptic. Fix $\theta \in (\theta_0, \pi)$. Then, the following properties are satisfied:
\begin{enumerate}[\rm(i)]
\item there exist constants $K \ge 1$, $\mu>0$,   (depending on $M$ and $\theta$) such that $\mu-A_p^0$ is an isomorphism from $W^{2,p}(\Rd;\Cm)$ into $L^p(\Rd;\Cm)$, which belongs to ${\mathcal P}(K,\pi-\theta)$.

\item there exists a constant $C\ge 1$ (depending on $M$ and $\theta$) such that
 \begin{equation}\label{p_a}\|(\mu-A_p^0)^{it}\|_{\mathcal{L}(L^p(\Rd;\mathbb{C}^m))}\le Ce^{\theta t}, \qquad\;\, t>0,
 \end{equation}
 and, consequently, $\theta_{\mu-A^0_p}\leq \theta_0$.
 \end{enumerate}

 \end{thm}

\begin{proof}
To begin with, we first recall that \cite[Theorem 9.4]{AHS} guarantees the existence of $\mu_1>0$ such that property (i) is satisfied with $\mu$ being replaced by $\mu_1$.
Since $\mu_1-A^0_p$ is an isomorphism from $W^{2,p}(\Rd;\Cm)$ into $L^p(\Rd;\Cm)$
we can show that $\|\cdot\|_{D(A^0_p)}$ is equivalent to the $W^{2,p}$-norm. Indeed, since the coefficients $Q^{hk}$ are bounded, it is immediate to estimate $\|\uu\|_{D(A^0_p)}\le C'\|\uu\|_{W^{2,p}(\Rd;\Rm)}$, for some positive constant $C'$ independent of $\uu$. On the other hand, we can estimate
\begin{align*}
\|\uu\|_{W^{2,p}(\Rd;\Cm)}&=\|(\mu_1-A^0_p)^{-1}(\mu_1-A^0_p)\uu\|_{W^{2,p}(\Rd;\Cm)}\\
&\le C'' \|(\mu_1-A^0_p)\uu\|_{L^{p}(\Rd;\Cm)}\\
&\le C''\|\uu\|_{D(A^0_p)}
\end{align*}
for some positive constants $C''$ independent of $\uu$.

By using this fact, observing that the resolvent set of $A^0_p-\mu_1$ contains the right-halfline and that $(\mu-A^0_p)^{-1}$ is continuous from $L^p(\Rd;\Cm)$ into $(D(A^0_p),\|\cdot\|_{D(A^0_p)})$ for any $\mu \ge\mu_1$, it follows that $\mu-A^0_p$ is an isomorphism from $W^{2,p}(\Rd;\Cm)$ into $L^p(\Rd;\Cm)$
for any $\mu\ge \mu_1$.

Further, thanks to \cite[Lemma 1.1]{AHS}, there exists $K(\theta)\ge 1$ such that the operator $\mu-A^0_p$ belongs to ${\mathcal P}(K(\theta),\pi-\theta)$ for any $\mu \ge \mu_1$.

Next, we note that \cite[Corollary 6.2]{Duo_Sim} yields the existence of $\mu_2>0$ such that estimate \eqref{p_a} holds true with $\mu$ being replaced by $\mu_2$. Moreover, \cite[Theorem 3]{Pru_Soh} implies that $\mu-A^0_p$ belongs to the class $BIP(L^p(\Rd; \Cm))$ with $\theta_{\mu-A^0_p} \le \theta_{\mu_2-A^0_p}$ for any $\mu\geq \mu_2$.
Consequently, in order to prove that properties in (i) and (ii) hold true with the same constant $\mu$, it is enough to take $\mu=\max\{\mu_1, \mu_2\}$.
 \end{proof}

 \begin{thm}\label{thm28}
Assume that the  matrices $Q^{hk}$ belong to $BUC(\Rd; \R^{m^2})$  and satisfy the Legendre-Hadamard ellipticity condition, namely  there exists a positive constant $C_1$ such that
\begin{equation}
\label{son}
{\rm Re}(A_0(x,\xi)\eta, \eta)\geq C_1|\eta|^2
\end{equation}
for any $x,\xi\in \Rd$ with $|\xi|=1$ and $\eta\in\Cm$.
Then, there exists a positive constant $\mu$ (depending on $m,d, C_1$ and the sup-norm of the coefficients $Q^{hk}$) such that  the following properties are satisfied:
\begin{enumerate}[\rm(i)]
\item
$\mu-A_p^0$ is sectorial with spectral angle $\varphi_{\mu-A_p^0}<\frac{\pi}{2}$. In particular, $\mu-A_p^0$ is closed and $0\in\rho(\mu-A_p^0)$;
\item
$\mu-A_p^0$ belongs to $BIP(L^p(\Rd;\mathbb{C}^m))$ and $\theta_{\mu-A_p^0}<\frac{\pi}{2}$;
\item
$\mu-A_p^0$ is invertible and $(\mu-A_p^0)^{-1}\in \mathcal{L}(L^p(\Rd;\mathbb{C}^m));W^{2,p}(\Rd;\mathbb{C}^m))$.
\end{enumerate}
\end{thm}

\begin{proof}
Thanks to Theorem \ref{thm_AHS}, the proof reduces to proving that $-\mathcal \A^0$ is uniformly $(M,\theta_0)$-elliptic for some $M>0$ and $\theta_0 \in [0,\pi/2)$. In this case, since $\varphi_{\mu-A_p^0}\le \theta_{\mu-A_p^0}\le \theta_0$, the assertions (i)-(iii) follow immediately.

We first show that, for any $x,\xi\in \Rd$ with $|\xi|=1$,
the spectrum of $A_0(x,\xi)$ is contained in $\Sigma_{\theta}$, where
$\theta=\arctan \left(\frac{m^2C_2}{C_1}\right)$ and $C_2=\max_{h,k=1,\ldots,d}\|Q^{hk}\|_{\infty}$.
For this purpose, we begin by observing that $|A_0(x,\xi)| \le d C_2|\xi|^2$ for any $x,\xi\in \Rd$.
Now, let us fix $x,\xi\in \Rd$ with $|\xi|=1$ and   $\lambda \in \sigma(A_0(x,\xi))$. Then, there exists $u\in \mathbb C^m$, with $|u|=1$, such that $\lambda u=A_0(x,\xi)u$. Multiplying scalarly by $u$ the previous equality and taking the real part, by \eqref{son} we deduce that ${\rm Re}(\lambda)\geq C_1.$
Moreover, $|{\rm Im}(\lambda)|\leq |\lambda|=|(A_0(x,\xi)u,u)|\leq |(A_0(x,\xi))|\le d C_2$.
Hence,
\begin{align*}
|{\rm arg}(\lambda) | = \left\vert\arctan\bigg (\frac{{\rm Im}\lambda}{{\rm Re}\lambda}\bigg )\right\vert =\arctan\bigg (\left\vert\frac{{\rm Im}\lambda}{{\rm Re}\lambda} \right\vert\bigg )  \leq \arctan \left(\frac{dC_2}{C_1}\right).
\end{align*}

Since $\sigma(A_0(x,\xi))\subset \{z \in \mathbb{C}: |z|\ge C_1\}$, by the formula for the inverse of a matrix, we deduce that there exists a positive constant $c_m$ (depending only on $m$) such that
\begin{align*}
|(A_0(x,\xi))^{-1}|\leq c_m|A_0(x,\xi)|^{m-1}C_1^{-m}\le c_m(d C_2)^{m-1}C_1^{-m}.
\end{align*}
Consequently, setting $M=\max\{C_2, c_m(dC_2)^{m-1}C_1^{-m}\}$ and  $\theta_0=\arctan\left( \frac{d C_2}{C_1}\right)$ we get the claim.
\end{proof}

Now, we focus our attention to the potential term of $\A$ proving some relevant properties of its realization in $L^p(\Rd;\Cm)$. To this aim, for any $p \in (1,\infty)$ and any matrix-valued function $V:\Rd\to \R^{m^2}$, we define the operator $V_p$ on $L^p(\Rd;\mathbb{C}^m)$ by setting $D(V_p)=\{\uu\in L^p(\Rd;\mathbb{C}^m):V\uu\in L^p(\Rd;\mathbb{C}^m)\}$ and $V_p \uu:=V\uu$.

\begin{thm}\label{finale_V}
Let assume that $V:\Rd\rightarrow \R^{m^2}$ is a matrix-valued function with entries in $W^{2,\infty}_{{\rm loc}}(\Rd)$ satisfying
\begin{equation}
(V(x)\eta,\eta) \ge c_0|\eta|^2,\qquad\;\,x\in\Rd,\;\,\eta\in\R^m,
\label{hyp-c0-prop}
\end{equation}
for some positive constant $c_0$.
Further, assume that there exists $\omega_0\in [0,\frac{\pi}{2})$ such that
\begin{equation}\label{ana_V-prop}
|{\rm Im}(V(x)\eta, \eta)| \leq \tan(\omega_0) {\rm Re}(V(x)\eta, \eta), \qquad\;\, x\in\Rd,\;\, \eta\in\Cm.
\end{equation}
Then, the following properties are satisfied:
\begin{enumerate}[\rm (i)]
\item
$\rho(-V_p)$ contains $\Sigma_{\pi-\omega_0}$ and,
for every $\omega'\in (\omega_0,\pi)$, there exists a positive constant $C_{\omega'}$ such that
 \begin{equation}\label{res}
 \|(\lambda +V_p)^{-1}\|_{\mathcal{L}(L^p(\Rd;\mathbb{C}^m))}\le \frac{C_{\omega'}}{1+|\lambda|},\qquad\;\,\lambda\in\Sigma_{\pi-\omega'}\cup\{0\};
 \end{equation}
 \item
the operator $V_p$ belongs to $BIP(L^p(\Rd;\mathbb{C}^m))$ and satisfies the estimate
\begin{eqnarray*}
\|V_p^{is}\|_{\mathcal{L}(L^p(\Rd;\Cm))}\le e^{\omega_0|s|}, \qquad\; \, s \in \R.
\end{eqnarray*}
In particular, the power angle $\theta_{V_p}$ does not exceed $\omega_0$.
\end{enumerate}
\end{thm}

\begin{proof}
(i) It is immediate to check that for every $x\in\Rd$, $\sigma(V(x))=P\sigma(V(x))\subset \overline{\Sigma_{\omega_0}} \cap \{z\in\mathbb{C}: {\rm Re} z\geq c_0\}$ and, consequently, $\Sigma_{\pi-\omega_0}\subseteq \rho(-V(x))$. Moreover, fix $\omega'\in \left (\omega_0,\frac{\pi}{2}\right )$ and $\lambda=|\lambda|e^{\pm i\theta}\in\Sigma_{\pi-\omega'}$. If $\theta\in [0,\frac\pi 2-\omega_0]$, then  $d(\lambda,\sigma(-V(x)))\ge |\lambda|$, whereas, if
$ \theta\in (\frac \pi 2-\omega_0, \pi-\omega']$, then
\begin{eqnarray*}
d(\lambda,\sigma(-V(x)))\ge d(\lambda, -\Sigma_{\omega_0}) \ge |\lambda|\sin (\omega'-\omega_0),
\end{eqnarray*}
where $d(\lambda,S)$ denotes the distance of $\lambda$ from the set $S$.
Consequently,
\begin{equation}\label{stimV2}
|(\lambda+V(x))^{-1}| \le \frac{1}{|\lambda| \sin (\omega'-\omega_0)}.
\end{equation}


Now, we fix $\f\in L^p(\R^d;\Cm)$, $\lambda\in\Sigma_{\pi-\omega'}$ for some $\omega'>\omega_0$ and consider the equation
$\lambda\uu+V_p\uu=\f$. We note for almost every $x\in\Rd$, it holds that $\uu(x)=(\lambda+V(x))^{-1}\f$.
From \eqref{stimV2} it follows that
\begin{eqnarray*}
|(\lambda+V(x))^{-1}\f(x)|\le \frac{1}{|\lambda|\sin(\omega'-\omega_0)}|\f(x)|   \end{eqnarray*}
for almost every $x\in\Rd$, so that
the function $(\lambda+V)^{-1}\f$ belongs to $L^p(\Rd;\Cm)$ and
\begin{eqnarray*}
\|\uu\|_{L^p(\Rd;\Cm)}\leq \frac{\|\f\|_{L^p(\Rd;\Cm)}}{|\lambda| \sin (\omega'-\omega_0)}, \qquad\;\, \lambda\in\Sigma_{\pi-\omega'}.
\end{eqnarray*}
The arbitrariness of $
\omega'>\omega_0$ implies that
all the sector $\Sigma_{\pi-\omega_0}$ lies in the resolvent set of the operator $-V_p$.

To complete the proof of \eqref{res}, it is enough to observe that $0\in \rho(-V(x))$ and $|V(x)^{-1}|\leq \frac{1}{d(0, \sigma(-V(x)))} \leq \frac{1}{c_0}$.
Arguing as above, we get that $0\in\rho(-V_p)$. Since the resolvent set is open, it contains a ball $B(r)$. Moreover, the resolvent operator is uniformly bounded in such a ball, since the function $\lambda\mapsto R(\lambda,-V_p)$ is continuous in $\rho(-V_p)$. Noticing that the function $\lambda\mapsto\dfrac{|\lambda|+1}{|\lambda|}$ is bounded in $\Sigma_{\pi-\omega'}\setminus B(r)$, from all the above results, estimate \eqref{res} follows.

(ii) We begin by observing that, since for any $x \in \Rd$ and any $\omega'\in (\omega_0,\frac{\pi}{2})$, $V(x)$ is m-$\omega'$ accretive (see \cite[Pag. 173]{Haas}), then $e^{\pm i (\frac{\pi}{2}-\omega')} V(x)$ is m-accretive (see \cite[Proposition 7.1.1]{Haas}). Consequently, applying \cite[Corollary 7.1.8]{Haas}, with $A=e^{\pm i (\frac{\pi}{2}-\omega')} V(x)$,  we obtain
\begin{equation*}
|(e^{  i (\frac{\pi}{2}-\omega')} V(x))^{is}|\le e^{-\frac{\pi}{2}s}, \qquad\; \, s<0
\end{equation*}
and
\begin{equation*}
|(e^{ -i (\frac{\pi}{2}-\omega')} V(x))^{is}|\le e^{\frac{\pi}{2}s}, \qquad\; \, s>0.
\end{equation*}
If we prove that
\begin{equation}\label{molt}
(e^{ \pm i(\frac{\pi}{2}-\omega')} V(x))^{is}=e^{\mp(\frac{\pi}{2}-\omega')s} (V(x))^{is}, \qquad\;\, x \in \Rd, \;\,s \in \Rm,
\end{equation}
then we get that
\begin{equation*}
e^{  \mp ( \frac{\pi}{2}-\omega')s} |V(x))^{is}|\le e^{\frac{\pi}{2}|s|}, \qquad\; \, s\in\R.
\end{equation*}
Hence,
\begin{equation*}
| (V(x))^{is}|\le e^{\pm \big(\frac\pi 2 -\omega'\big) s +  \frac{\pi}{2}|s|}, \qquad\; \, s\in\R
\end{equation*}
and, by dividing the cases in which $s>0$ and $s<0$, we get that
\begin{equation}\label{1'}
|(V(x))^{is}|\le e^{\omega'|s|}, \qquad\; \, s\in\mathbb R.
\end{equation}

To prove the equality \eqref{molt}, we observe that  $\sigma(V(x))$ is a compact subset of $\Sigma_{\omega'} \cap \{ |z|>c'\}$ with $0<c'<c_0$. Let $U\subseteq \Sigma_{\omega'} \cap \{ z\in\mathbb C: |z|>c'\}$ be a bounded open neighborhood of    $\sigma(V(x))$ with smooth, positively oriented boundary $+\partial U$.
Then $e^{\pm i(\frac \pi 2 - \omega')}U $ is a bounded open neighborhood of  $ e^{\pm i(\frac \pi 2 - \omega')}\sigma(V(x))=\sigma(e^{\pm i(\frac \pi 2 - \omega')}V(x))$ with smooth, positively oriented boundary $e^{\pm i(\frac \pi 2 - \omega')}(+\partial U)$ and such that $e^{\pm i(\frac \pi 2 - \omega')}U \subseteq \Sigma_{\frac \pi 2} \cap\{z\in\mathbb C: |z|>c'\} \subseteq\mathbb C\setminus (-\infty,0]$. Consequently, since the function $z\mapsto z^{is}$ is holomorphic in $\mathbb C\setminus (-\infty,0]$, by the Dunford-Schwartz functional calculus (see e.g. \cite[Sec. VII.3]{DS}), it follows that

\begin{align*}
(e^{\pm i(\frac \pi 2 - \omega')}V(x))^{is}=&\frac{1}{2\pi i}\int_{e^{\pm i(\frac \pi 2 -\omega')}(+\partial U)} \lambda^{is} (\lambda-e^{\pm i(\frac \pi 2 - \omega')}V(x))^{-1}d\lambda\\
=&\frac{e^{\mp i(\frac \pi 2 -\omega')}}{2\pi i}\int_{e^{\pm i(\frac \pi 2 -\omega')}(+\partial U)} \lambda^{is} (\lambda e^{\mp i(\frac \pi 2 -\omega')}-V(x))^{-1}d\lambda .
\end{align*}
By performing the  change of variable $z=e^{\mp i(\frac \pi 2 -\omega')}\lambda$ and taking into account that  $-\frac \pi 2 <{\rm arg}(z)\pm (\frac\pi 2 -\omega')<\frac\pi 2$, we get that
\begin{align*}
(e^{\pm i(\frac \pi 2 - \omega')}V(x))^{is}=\frac{1}{2\pi i}\int_{+\partial U} \big(z e^{\pm i(\frac \pi 2 -\omega')}\big)^{is} (z-V(x))^{-1}dz=
e^{\mp (\frac \pi 2 -\omega')s} \big(V(x)\big)^{is}.
\end{align*}

Formula \eqref{molt} is proved.

Finally, we observe that, since $((V_p)^{is}\f)(x)=(V(x))^{is}\f(x)$, for any $\f\in L^p(\Rd;\Cm)$  and almost every $x\in\Rd$, from \eqref{1'} we infer that
\begin{eqnarray*}
\|(V_p)^{is}\|_{\mathcal{L}(L^p(\Rd;\mathbb{C}^m))}\le e^{\omega'|s|}
\end{eqnarray*}
for any $s\in \R$, whence $\theta_{V_p}\le\omega_0$.
\end{proof}

\section{The generation result in $L^p(\Rd; \mathbb{C}^m)$}

Here, for every $p \in (1,\infty)$ we set $A_p:=A_p^0-V_p$ and show that the operator $-A_p$,
defined on $D(A_p):=D(A_p^0)\cap D(V_p)$, generates a strongly continuous and analytic semigroup in $L^p(\R^d;\Cm)$. This is done under the following assumptions on the matrix-valued functions $Q^{hk}$.

\begin{hyp0}\label{H1}
For any $h,k=1, \ldots,d$ the matrices $Q^{hk}$ belong to $BUC(\Rd;\R^{m^2})$ and satisfy the Legendre-Hadamard ellipticity condition (see \eqref{son}).
\end{hyp0}

As far as the matrix-valued function $V$ is concerned, we assume one of the following alternative sets of hypotheses.

\begin{hyp}\label{H1-1}
\begin{enumerate}[\rm (i)]
\item
$V:\Rd\rightarrow \R^{m^2}$ is a matrix-valued function with entries in $W^{2,\infty}_{{\rm loc}}(\Rd)$ and there exists a positive constant $c_0$ such that
\begin{equation}
(V(x)\eta,\eta) \ge c_0|\eta|^2,\qquad\;\,x\in\Rd,\;\,\eta\in\R^m.
\label{hyp-c0}
\end{equation}
Moreover, there exists $\omega_0\in [0,\frac{\pi}{2})$ such that
\begin{equation}\label{ana_V}
|{\rm Im}(V(x)\eta, \eta)| \leq \tan(\omega_0) {\rm Re}(V(x)\eta, \eta), \qquad\;\, x\in\Rd,\;\, \eta\in\Cm;
\end{equation}
\item
$VQ^{hk}=Q^{hk}V$ for every $h,k\in\{1, \dots, d\}$ and
there exists $\alpha\in (0,1)$ such that  $(D_kV)V^{-\alpha}$ and $(D_{hk}V)V^{-\alpha}$ are bounded
 for any $h,k=1, \dots, d$.
\end{enumerate}
\end{hyp}

\begin{hyp}\label{hyp_2}
\begin{enumerate}[\rm(i)]
\item
There exist a function $v\in W^{2, \infty}_{\rm loc}(\Rd;\R)$, with positive infimum $v_0$, and $\alpha \in (0,1)$ such that $v^{-\alpha} D_jv$ and $v^{-\alpha}D_{ij}v$ are bounded in $\Rd$ for any $i,j=1, \ldots,d$;
\item
$V:\Rd\rightarrow \R^{m^2}$ is a matrix-valued function with entries in $L^\infty_{{\rm loc}}(\Rd)$ such that $|v_{ij}(x)|\le c (v(x))^{\gamma}$ for any $i,j=1, \ldots,d$, with $i\neq j$, and $|v_{ii}(x)-v(x)|\le c (v(x))^{\gamma}$ for any $i=1, \ldots,d$, $x \in \Rd$, some positive constant $c$ and some $\gamma \in (0,1)$.
\end{enumerate}
\end{hyp}

\begin{thm}\label{main}
Assume Hypothesis $\ref{H1}$ and either Hypotheses $\ref{H1-1}$ or $\ref{hyp_2}$. Then, for every $p \in (1,\infty)$ the operator $(A_p, D(A_p))$ generates a strongly continuous and analytic semigroup in $L^p(\R^d;\Cm)$. Moreover, there exist two positive constants $c_1$ and $c_2$ such that
\begin{align}\label{norm_eq}
c_1\big (\|\uu\|_{W^{2,p}(\Rd;\Cm)}+\|\uu\|_{D(V_p)}\big )\le \|\uu\|_{D(A_p)}\le c_2\big (\|\uu\|_{W^{2,p}(\Rd;\Cm)}+\|\uu\|_{D(V_p)}\big )
\end{align}
for any $\uu\in D(A_p)$.
\end{thm}

\begin{proof}
We split the proof into three steps. In the first two steps, we assume Hypotheses \ref{H1-1} on $V$ and in Step 3, we replace these assumptions with Hypotheses \ref{hyp_2}.

{\em Step 1}. Here, we prove that
\begin{align}\label{muro}
&[A_p^0(\nu+V_p)^{-1}-(\nu+V_p)^{-1}A_p^0]\uu\notag\\
=&-\sum_{h,k=1}^d(\nu+V_p)^{-1}(Q^{hk}+Q^{kh})D_kV (\nu+V_p)^{-1}D_h\uu \notag\\
 &+\sum_{h,k=1}^d(\nu+V_p)^{-1}(Q^{hk}+Q^{kh}) D_hV(\nu+V_p)^{-1}D_kV (\nu+V_p)^{-1} \uu\notag\\
 &-\sum_{h,k=1}^d(\nu+V_p)^{-1}Q^{hk}D_{hk}V(\nu+V_p)^{-1} \uu
\end{align}
for any $\uu\in W^{2,p}(\Rd;\mathbb{C}^m)$ and $\nu\in \Sigma_{\pi-\omega_0}$.

We start considering a matrix-valued function $M:\Rd\rightarrow \R^{m^2}$ with entries in $W^{2,\infty}_{\rm loc}(\Rd)$ such that $MQ^{hk}=Q^{hk}M$ for any $h,k=1, \dots, d$. Then,
\begin{align}
[\A^0M-M\A^0]\uu=&\sum_{h,k=1}^dQ^{hk}D_{hk}(M\uu) - \sum_{h,k=1}^dM(Q^{hk}D_{hk}\uu)\notag\\
=&\sum_{h,k=1}^dQ^{hk}(D_{hk}M\uu+ D_kM D_h\uu + D_hM D_k\uu) +  \sum_{h,k=1}^d(Q^{hk}M-MQ^{hk})D_{hk}\uu\notag\\
=&\sum_{h,k=1}^d(Q^{hk}+Q^{kh})D_kMD_h\uu+ \sum_{h,k=1}^dQ^{hk}D_{hk}M\uu.
\label{prel}
\end{align}
\noindent
Now, let us fix $\nu \in \Sigma_{\pi-\omega_0}$. In order to apply the previous equality with $M=(\nu+V)^{-1}$, we observe that, by differentiating the equality $I=(\nu+V)(\nu+V)^{-1}$, we obtain that
\begin{align*}
&D_h(\nu+V)^{-1}=-(\nu+V)^{-1} D_hV (\nu+V)^{-1}\\
 &D_{hk}(\nu+V)^{-1}
 =(\nu+V)^{-1}\big [D_hV(\nu+V)^{-1}D_kV- D_{hk}V +
D_kV(\nu+V)^{-1}D_hV\big ](\nu+V)^{-1}.
\end{align*}
Consequently, applying \eqref{prel}, with $M$ being replaced by $(\nu+V)^{-1}$, and observing that $Q^{hk}(\nu+V)^{-1}=(\nu+V)^{-1}Q^{hk}$, since $Q^{hk}V=VQ^{hk}$ for every $h,k=1,\ldots,d$, we deduce that
\begin{align*}
 &[\A^0(\nu+V)^{-1}-(\nu+V)^{-1}\A^0]\uu\\
 =&-\sum_{h,k=1}^d(Q^{hk}+Q^{kh})(\nu+V)^{-1}D_kV (\nu+V)^{-1}D_h\uu\\
 &+\sum_{h,k=1}^dQ^{hk}(\nu+V)^{-1}\Big( D_hV(\nu+V)^{-1}D_kV-D_{hk}V +
D_kV(\nu+V)^{-1}D_hV\Big) (\nu+V)^{-1} \uu.
\end{align*}
To conclude the proof of \eqref{muro} it suffices to show that $(\mu+V_p)^{-1}(W^{2,p}(\Rd;\Cm))\subset W^{2,p}(\Rd;\Cm)$. Thus, let us fix $\f\in W^{2,p}(\Rd;\Cm)$, $i,j\in\{1,\ldots,d\}$ and observe that
\begin{align}\label{der_1}
D_j[(\nu+V_p)^{-1}\f]=-(\nu+V)^{-1}D_j V(\nu+V)^{-1}\f+(\nu+V)^{-1}D_j \f
\end{align}
and
\begin{align}\label{der_2}
&D_{ij}[(\nu+V_p)^{-1}\f]\notag\\
=& (\nu+V)^{-1}\big[D_iV(\nu+V)^{-1}D_jV- D_{ij}V +
D_jV(\nu+V)^{-1}D_iV\big ] (\nu+V)^{-1}\f\notag\\
+&(\nu+V)^{-1}D_{ij}\f-(\nu+V)^{-1}D_iV(\nu+V)^{-1}D_j \f-(\nu+V)^{-1}D_jV(\nu+V)^{-1}D_i \f.
\end{align}
Using the assumptions on $V$, Theorems \ref{power_thm}(iii) and \ref{finale_V}, we
can show that $D_j(\nu+V_p)^{-1}\f$ belongs to $L^p(\Rd;\Cm)$ and
 \begin{align}\label{serve}
 \|D_jV(\nu+V_p)^{-1}\|_{\mathcal{L}(L^p(\Rd;\Cm))}&= \|D_j V V^{-\alpha}V^{\alpha}(\nu+V_p)^{-1}\|_{\mathcal{L}(L^p(\Rd;\Cm))}\notag\\
 &\le C\|V_p^\alpha (\nu+V_p)^{-1}\|_{\mathcal{L}(L^p(\Rd;\Cm))}\notag\\
 &\le \frac{C(\gamma,\alpha)}{(1+|\nu|)^{1-\alpha}}
 \end{align}
 for any $\nu \in \Sigma_{\pi-\gamma}$, any $\gamma>\varphi_{V_p}$ and some positive constant $C(\gamma,\alpha)$.
 Arguing similarly, we can prove that
 \begin{equation}\label{serve1}
\|D_{ij}V(\nu+V_p)^{-1}\|_{\mathcal{L}(L^p(\Rd;\mathbb{C}^m))}\le \frac{C'(\gamma,\alpha)}{(1+|\nu|)^{1-\alpha}}
 \end{equation}
 for any $\nu$ and $\gamma$ as above and some positive constant $C'(\gamma,\alpha)$,
 which together with \eqref{der_2} and \eqref{serve}, yields that $D_{ij}(\nu+V_p)^{-1}\f \in L^p(\Rd;\Cm)$.

{\em Step 2}. Here, to complete the proof in this case, we apply Theorem \ref{M_P} with $L_1=\mu-A^0_p=:B_p^\mu$, where $\mu$ is given by Theorem \ref{thm28}, and $L_2=V_p$.
Thanks to Theorems \ref{thm28} and \ref{finale_V} it is immediate to check that $B_p^\mu$ and $V_p$ satisfy all the assumptions in Theorem \ref{M_P} but the commutator estimate \eqref{commutator_0} that we prove here below. To this purpose, we fix $\gamma_1 \in (\theta_{B_p^{\mu}}, \frac{\pi}{2})$, $\gamma_2 \in (\theta_{V_p},\frac{\pi}{2})$ and prove that
\begin{equation}
\label{commutator}
\|C(\lambda, \nu)\|_{\mathcal{L}(L^p(\Rd;\Cm))}\le \frac{c}{(1+|\lambda|)|\nu|^{2-\alpha}},\qquad\;\, \lambda \in \Sigma_{\pi -\gamma_1},\,\;\nu \in \Sigma_{\pi-\gamma_2}
\end{equation}
for some constants $c>0$ and $\alpha$ given by Hypothesis \ref{H1-1}(ii),
where
\begin{align*}
C(\lambda, \nu):=B_p^\mu(\lambda+B_p^\mu)^{-1}[(B_p^\mu)^{-1}(\nu+V_p)^{-1}-(\nu+V_p)^{-1}(B_p^\mu)^{-1}].
\end{align*}

Let us observe that, for any $\f\in L^p(\Rd;\Cm)$ it holds that
\begin{align*}
C(\lambda, \nu)\f&=B_p^\mu(\lambda+B_p^\mu)^{-1}(B_p^\mu)^{-1}[(\nu+V_p)^{-1}B_p^\mu-B_p^\mu(\nu+V_p)^{-1}](B_p^\mu)^{-1}\f\\
&=(\lambda+B_p^\mu)^{-1}[(\nu+V_p)^{-1}B_p^\mu-B_p^\mu(\nu+V_p)^{-1}](B_p^\mu)^{-1}\f\\
&= (\lambda+B_p^\mu)^{-1}[A_p^0(\nu+V_p)^{-1}-(\nu+V_p)^{-1}A_p^0](B_p^\mu)^{-1}\f.
\end{align*}
Thanks to \eqref{muro}, we deduce that
\begin{align*}
&C(\lambda, \nu)\f\\
=&-(\lambda+B_p^\mu)^{-1}(\nu+V_p)^{-1}\sum_{h,k=1}^d(Q^{hk}+Q^{kh})D_kV (\nu+V_p)^{-1}D_h(B_p^\nu)^{-1}\f \notag\\
 &+(\lambda+B_p^\mu)^{-1}(\nu+V_p)^{-1}\bigg(\sum_{h,k=1}^d(Q^{hk}+Q^{kh}) D_hV(\nu+V_p)^{-1}D_kV\bigg) (\nu+V_p)^{-1}(B_p^\mu)^{-1}\f\\
 & -(\lambda+B_p^\mu)^{-1}(\nu+V_p)^{-1}\bigg(\sum_{h,k=1}^dQ^{hk}D_{hk}V\bigg)(\nu+V_p)^{-1}(B_p^\mu)^{-1}\f\\
 =&\!:\sum_{i=1}^3C_i(\lambda, \nu)\f.
 \end{align*}

Taking Theorems \ref{thm28}, \ref{finale_V} and estimate \eqref{serve} into account and using the boundedness of the coefficients $Q^{hk}$, we can estimate
 \begin{align}
&\|C_1(\lambda,\nu)\f\|_{L^p(\Rd;\Cm)}\notag\\
\le &\|(\lambda+B_p^{\mu})^{-1}\|_{\mathcal{L}(L^p(\Rd;\Cm))}\|(\nu+V_p)^{-1}\|_{\mathcal{L}(L^p(\Rd;\Cm))}\notag\\
&\qquad\times\sum_{h,k=1}^d\|Q^{hk}+Q^{kh}\|_{\infty}\|D_hV(\nu+V_p)^{-1}\|_{\mathcal{L}(L^p(\Rd;\Cm))}
 \|(B_p^{\mu})^{-1}\f\|_{W^{1,p}(\Rd;\Cm)}\notag\\
\le &\frac{c_1}{(1+|\nu|)^{2-\alpha}(1+|\lambda|)}\|\f\|_{L^p(\Rd;\Cm)}
\label{estim-C1}
 \end{align}
 for every $\lambda\in\Sigma_{\pi-\gamma_1}$, every $\nu\in \Sigma_{\pi-\gamma_2}$ and some positive constant $c_1$ independent of $\lambda$ and $\nu$.
 In a similar way, taking also \eqref{serve1} into account, we can estimate
 \begin{align}
\|C_2(\lambda, \nu)\f\|_{L^p(\Rd;\Cm)}\le \frac{c_2}{(1+|\nu|)^{3-2\alpha}(1+|\lambda|)}\|\f\|_{L^p(\Rd;\Cm)}
\label{estim-C2}
\end{align}
and
 \begin{align}
\|C_3(\lambda, \nu)\f\|_{L^p(\Rd;\Cm)}\le \frac{c_3}{(1+|\nu|)^{2-\alpha}(1+|\lambda|)}\|\f\|_{L^p(\Rd;\Cm)}
\label{estim-C3}
\end{align}
for the same values of $\lambda$ and $\nu$ as above and some positive constants $c_2$ and $c_3$ independent of $\lambda$ and $\nu$.
From \eqref{estim-C1}-\eqref{estim-C3}, estimate \eqref{commutator} follows.

Applying Theorem \ref{M_P} with $\beta=0$ and $\delta=1-\alpha$, we conclude that the operator $-A_p$ is quasi-sectorial and the spectral angle of the operator $-B_p^{\mu}+V_p$ does not exceed the maximum between $\gamma_1$ and $\gamma_2$. Since both these constants are less than $\frac{\pi}{2}$, the spectral angle is itself less than $\frac{\pi}{2}$, so that the operator $A_p$ generates an analytic semigroup in $L^p(\Rd;\Cm)$, which is also strongly continuous since $D(A_p)$ is dense in $L^p(\Rd;\Cm)$.

To conclude this step, we prove \eqref{norm_eq}. The first inequality is immediate. To prove the other inequality, we first observe that, since $A_p$ is a closed operator, $(D(A_p),\|\cdot\|_{D(A_p)})$ is a Banach space. Analogously, the closedeness of $A^0_p$ and $V_p$, together with the equivalence of the norms of $D(A^0_p)$ and the classical $W^{2,p}(\Rd;\Cm)$-norm proved in the proof of Theorem \ref{thm_AHS}, imply that $(D(A_p), \|\cdot\|_{W^{2,p}(\Rd;\Cm)}+\|\cdot\|_{D(V_p)})$ is a Banach space too. 
The open mapping theorem, applied to the identity operator from $ (D(A_p), \|\cdot\|_{W^{2,p}(\Rd;\Cm)}+\|\cdot\|_{D(V_p)})$ and $(D(A_p), \|\cdot\|_{D(A_p)})$ yields the second estimate in \eqref{norm_eq}.

{\em Step 3.} Here, we assume that $V$ satisfies Hypotheses \ref{hyp_2}. First of all, we observe that the matrix-valued function $vI$ satisfies Hypotheses \ref{H1-1}. Consequently, by Steps 1 and 2, the realization $L_p$ of $\A^0-vI$ in $L^p(\Rd;\Cm)$, with domain $D_p=\{\uu \in W^{2,p}(\Rd;\Cm): v\uu \in L^p(\Rd;\Cm)\}$, generates a strongly continuous analytic semigroup.

Next, we prove that the operator $ B:D_p\to L^p(\Rd;\Cm)$, defined by $B\uu=-V\uu+v\uu$ for any $\uu \in D_p$, is $L_p$-bounded with $L_p$-bound equal to $0$. Once this property is proved, applying \cite[Theorem 2.10]{engnagel} we conclude that the realization $A_p$ of the operator $\A$ in $L^p(\Rd;\Cm)$, with domain $D_p$, generates a strongly continuous analytic semigroup.

For this purpose, let us fix $\uu \in D_p$. By using Hypotheses \ref{hyp_2}(ii) we can find a constant $\tilde{c}>0$ and, for any $\varepsilon>0$,
 a positive constant $C_{\varepsilon,\gamma}$ such that
\begin{equation*}
\|B\uu\|_{L^p(\Rd;\Cm)}\le  \tilde{c}\|v^{\gamma}\uu\|_{L^p(\Rd;\Cm)}
\le  \varepsilon \|v\uu\|_{L^p(\Rd;\Cm)}+C_{\varepsilon,\gamma}\|\uu\|_{L^p(\Rd;\Cm)}.
\end{equation*}
Using \eqref{norm_eq} we can go further in the previous estimate to infer that
\begin{align*}
\|B\uu\|_{L^p(\Rd;\Cm)}\le & \frac{\varepsilon}{c_1}\|\uu\|_{D(A_p^0-vI)} +C_{\varepsilon,\gamma}\|\uu\|_{L^p(\Rd;\Cm)}\\
\le &\frac{\varepsilon}{c_1}\|(A_p^0-v)\uu\|_{L^p(\Rd;\Cm)}+ \bigg(C_{\varepsilon,\gamma}+\frac{\varepsilon}{c_1}\bigg)\|\uu\|_{L^p(\Rd;\Cm)}
\end{align*}
and this shows that the operator $B$ is $L_p$-bounded with $L_p$-bound equal to $0$.

Finally, we show that $D_p=\{\uu\in W^{2,p}(\Rd;\Cm): V\uu\in L^p(\Rd;\Cm)\}$. For this purpose, it suffices to prove that $\uu\in D(V_p)$ if and only if $v\uu\in L^p(\Rd;\Cm)$ and
there exist two positive constants $K_1$ and $K_2$ such that
\begin{equation}
K_1\|v\uu\|_{L^p(\Rd;\Cm)}\le \|\uu\|_{D(V_p)}\le K_2\|v\uu\|_{L^p(\Rd;\Cm)}
\label{annalisa}
\end{equation}
for every $u\in D(V_p)$. The assumptions on $V$ in Hypotheses \ref{hyp_2} allow us to split
$V=vI+W$, where the entries of the matrix-valued function $W$ are functions whose moduli can be bounded from above in terms of the function $v^{\gamma}$. Therefore,
\begin{eqnarray*}
|v\uu|\le |V\uu|+|W\uu|\le |V\uu|+cv^{\gamma}|\uu|
\end{eqnarray*}
for some positive constant $c$. Now, using Young inequality, we can estimate
$cv^{\gamma}\le \frac{1}{2}v+c'$ so that
\begin{eqnarray*}
|v\uu|\le |V\uu|+\frac{1}{2}v|\uu|+c'|\uu|
\end{eqnarray*}
and the first part of estimate \eqref{annalisa} follows with $K_1=(\max\{2,2c'\})^{-1}$.

Arguing similarly, we can estimate
\begin{eqnarray*}
|V\uu|\le v|\uu|+|W\uu|\le v|\uu|+ cv^{\gamma}|\uu|
\le c_*v|\uu|+c_{**}|\uu|\le (c_*+c_{**}v_0^{-1})v|\uu|
\end{eqnarray*}
for some positive constants $c_*$ and $c_{**}$, so that the second part of \eqref{annalisa} follows with $K_2=c_*+(c_{**}+1)v_0^{-1}$.

Clearly, estimate \eqref{norm_eq} holds true also in this situation, since the proof provided in the last part of Step 2 does not depend on the form of the potential term $V$.
\end{proof}

\section{Further results}
In this section we assume that the hypotheses of Theorem \ref{main} are always satisfied.

\begin{prop}
$C^\infty_c(\Rd;\Cm)$ is a core for $A_p$.
\end{prop}

\begin{proof}
In view of the estimate \eqref{norm_eq}, we start proving that
 for any function  $\uu \in W^{2,p}(\Rd; \Cm)$ such that $V\uu\in L^p(\Rd;\Cm)$ there exists a sequence $(\vv_n)\subset W^{2,p}(\Rd;\Cm)$ with compact support such that $\|\vv_n-\uu\|_{W^{2,p}(\Rd;\Cm)}+\|V\vv_n-V\uu\|_{L^p(\Rd;\Cm)}$ vanishes as $n$ tends to $\infty$. Let us fix $\uu \in D(A_p)$ and a standard sequence $(\theta_n)$ of smooth cut-off functions. Then, the sequence $(\vv_n)$ defined by
$\vv_n:=\uu \theta_n$ for any $n \in \N$, has the required properties. To conclude the proof it suffices to approximate any compactly supported function in $W^{2,p}(\Rd;\Cm)$ by  functions in $C^\infty_c(\Rd;\Cm)$. This can be easily obtained regularizing the functions by convolution.
\end{proof}

For any $p\in (1,\infty)$, we denote by $(T_p(t))$ the semigroup generated by the operator $A_p$ in $L^p(\Rd;\Cm)$ and we prove the consistency of such semigroups on the $L^p$-scale.

\begin{prop}\label{prop-4.2}
For every $p,q\in (1,\infty)$,  the semigroups $(T_p(t))$ and $(T_q(t))$ coincide on $L^p(\Rd;\Cm)\cap L^q(\Rd;\Cm)$.
\end{prop}

\begin{proof}
Fix $p,q\in (1,\infty)$ and observe that to get the claim it suffices to prove that the resolvent operators are consistent, i.e., that $R(\lambda,A_p)=R(\lambda,A_q)$ on $L^p(\Rd;\Cm)\cap L^q(\Rd;\Cm)$ for any $\lambda \in \rho(A_p)\cap \rho(A_q)$. Indeed, by the representation of the semigroups in $L^p(\Rd;\Cm)$ and $L^q(\Rd;\Cm)$ in terms of the resolvent operators via the Dunford integral, the equivalence $T_p(t)=T_q(t)$ on $L^p(\Rd;\Cm)\cap L^q(\Rd;\Cm)$ follows for every $t>0$.

{\em Step 1}.
For every $t\in [0,1]$, we consider the operator $\A_t$ formally defined by
\begin{eqnarray*}
\A_t\uu=\displaystyle\sum_{h,k=1}^dQ^{hk}_tD_{hk}\uu-V\uu=:\A^0_t\uu-V\uu,\qquad\;\,t\in [0,1],
\end{eqnarray*}
on smooth enough functions $\uu:\Rd\to\Cm$, where $Q^{hk}_t=tQ^{hk}+(1-t)\delta_{h,k}I$ for every $t\in [0,1]$.
Since, the coefficients $Q^{hk}_t$ satisfy the assumptions in Theorem \ref{thm28} with the constant $C_1$ being replaced by $C_1 \wedge 1$, the operator $\A^0_t$ verifies the assumptions of Theorem \ref{thm_AHS} with constants $M,\theta_0$ independent of $t$.
Consequently, applying  Theorem \ref{thm_AHS} we can find $\mu$ independent of $t\in [0,1]$ such that $\mu-(A_t)^0_p$ satisfies the properties stated in Theorem \ref{thm28} and there exists $\overline\theta\in[0,\frac{\pi}{2})$, independent of $t \in [0,1]$ (see also the proof of Theorem \ref{thm28}), and, consequently, $\overline\varphi \in [0,\frac{\pi}{2})$, independent of $t$, such that
\begin{eqnarray*}
\|R(\lambda,(A_t)^0_p-\mu)\|_{\mathcal{L}(L^p(\Rd;\Cm))}\le \frac{C}{1+|\lambda|}
\end{eqnarray*}
and
$\|((A_t)^0_p-\mu)^{is}\|_{\mathcal{L}(L^p(\Rd;\Cm))}\le \overline{M} e^{\overline \theta|s|}$
for any $\lambda \in \Sigma_{\pi-\overline \varphi}$, $s \in \R$ and some positive constants $C$ and $\overline{M}$ independent of $t$.

In view of this, the proof of Theorem \ref{main} shows that the constant $c$ in estimate \eqref{commutator} can be chosen independent of $t$. Thus, taking Theorem \ref{M_P} into account, we can find $\sigma$ independent of $t$ such that $\sigma+(A_t)_p$ is sectorial in a sector independent of $t$. We have so proved that the resolvent sets $\rho((A_t)_p)$ and $\rho((A_t)_q)$ contain right-halflines independent of $t$. This shows that there exists $\tilde{\lambda} \in \rho((A_t)_p))\cap\rho((A_t)_q)$ for any $t \in [0,1]$.

{\em Step 2}.
Here, we prove that
$R(\tilde\lambda,A_p)$
and $R(\tilde\lambda,A_q)$ coincide on $L^p(\Rd;\Cm)\cap L^q(\Rd;\Cm)$.
First, observe that by Theorem \ref{main},
there exists a positive constant $C_{p,q}(t)$,
such that
\begin{align}\label{ele}
&\|\uu\|_{W^{2,p}(\Rd;\Cm)}+\|V_p\uu\|_{L^p(\Rd;\Cm)}
+\|\uu\|_{W^{2,q}(\Rd;\Cm)}+\|V_q\uu\|_{L^q(\Rd;\Cm)}\notag\\
\le &C_{p,q}(t)\left (\|\widetilde\lambda\uu-\A_t\uu\|_{L^q(\Rd;\Cm)}
+\|\widetilde
\lambda\uu-\A_t\uu\|_{L^q(\Rd;\Cm)}\right )
\end{align}
for every $\uu\in D(A_p)\cap D(A_q)$.
Now, let us consider the Banach spaces $(X_{p,q},\|\cdot\|_{X_{p,q}})$ and $(Y_{p,q},\|\cdot\|_{Y_{p,q}})$, where
\begin{eqnarray*}
X_{p,q}:=D(A_p)\cap D(A_q), \qquad\;\,
Y_{p,q}:=L^p(\Rd;\Cm)\cap L^q(\Rd;\Cm),
\end{eqnarray*}
endowed with the norms
$\|\uu\|_{X_{p,q}}=\|\uu\|_{D(A_p)}+\|\uu\|_{D(A_q)}$ for every $\uu\in X_{p,q}$ and $\|\vv\|_{Y_{p,q}}=\|\vv\|_{L^{p}(\Rd;\Cm)}+\|\vv\|_{L^q(\Rd;\Cm)}$ for every $\vv\in Y_{p,q}$. Let us
define the family of the operators $\{L_t: t\in [0,1]\}$ by setting
 \begin{eqnarray*}L_t:= \tilde{\lambda}-\A_t:X_{p,q}\to Y_{p,q}, \qquad\;\, t \in [0,1].\end{eqnarray*}
Observe that if $t_n$ converges to some $t_0\in [0,1]$ as $n$ tends to $\infty$, then
for every $\uu\in X_{pq}$, it holds that
\begin{align*}
\|L_{t_n}\uu- L_{t_0}\uu\|_{Y_{p,q}}&= \|(t_0-t_n)(\A_0 \uu-\Delta \uu)\|_{Y_{p,q}}\\
&\leq C|t_n-t_0|(\|\uu\|_{W^{2,p}(\Rd;\Cm)}+\|\uu\|_{W^{2,q}(\Rd;\Cm)})\\
&\leq C|t_n-t_0|\|\uu\|_{X_{p,q}},
\end{align*}
where $C$ is a positive constant depending only on the $L^\infty$-norm of the coefficients $Q^{hk}$.
Hence, the map $t\mapsto L_t$  is continuous from $[0,1]$ into $\mathcal L(X_{pq}, Y_{pq})$.

Then, by estimate \eqref{ele} we get
\begin{eqnarray*}
\|L_t\uu\|_{Y_{p,q}}\ge (C_{p,q}(t))^{-1}\|\uu\|_{X_{p,q}}
\end{eqnarray*}
for any $\uu \in X_{p,q}$.
To conclude, we observe that, by the Trotter-Kato Product formula (see \cite[III Corollary 5.8]{engnagel}), the operator $L_0$ is consistent in the $L^p$-spaces and, consequently, it is surjective from $X_{p,q}$ into $Y_{p,q}$.
Thus, by applying the continuity method, in the stronger version of \cite[Lemma 8.1]{MNS}, we deduce that $L_1:X_{p,q}\to Y_{p,q}$ is surjective as well, and this completes the proof of this step.

{\em Step 3}.
Here, we fix $\lambda\in \rho(A_p)\cap \rho (A_q)$ and claim that $R(\lambda, A_p)= R(\lambda,A_q)$ on $L^p(\Rd;\Cm)\cap L^q(\Rd;\Cm)$.
Let us fix $\f\in L^p(\Rd;\Cm)\cap L^q(\Rd;\Cm) $
and consider the solution $\uu\in D(A_p)$ of the equation $\lambda \uu-\A\uu=\f$. By the Sobolev embedding theorem, $\uu$ belongs to $L^r(\Rd;\Cm)$ for some $r>p$. Clearly, if $r \ge q$ then $\uu$ belongs to $L^q(\Rd;\Cm)$ and by difference $\A\uu$ belongs to $L^q(\Rd;\Cm)$ as well. Now, we consider the equation
$\tilde\lambda \vv-\A\vv=\tilde\lambda \uu-\A\uu$. By Step 2, this equation admits a solution $\vv \in X_{p,q}$ and it coincides with $\uu$ by uniqueness in $L^p(\Rd;\Cm)$. Hence, the equality $R(\lambda, A_p)= R(\lambda,A_q)$ is proved in this case.

On the other hand, if $r<q$ then with the previous argument we obtain that $\uu \in D(A_r)$. Repeating this procedure, in a finite number of steps we show that $\uu$ belongs to $D(A_q)$, proving the claim.
\end{proof}

In view of Proposition \ref{prop-4.2}, in the rest of the paper we simply write $(T(t))$ instead of $(T_p(t))$.
As a next step, we show some inclusions for the spectra of the operators $A_p$ when $p$ varies in $(1,\infty)$. Clearly, $P\sigma(A_p)\subset A\sigma(A_p)$.

\begin{prop}
\label{prop-4.3}
For any  $1<p<q<\infty$ it holds that $P\sigma(A_p)\subseteq P\sigma(A_q)$ and $P\sigma(A_q)\subseteq A\sigma(A_p)$.
\end{prop}

\begin{proof}
Let us begin by proving the inclusion $P\sigma(A_p)\subseteq P\sigma(A_q)$. For this purpose, we fix $\lambda\in P\sigma(A_p)$ and let $\uu\in D(A_p)$, different from zero, be such that $\lambda\uu-A_p\uu=\bm{0}$. Then, $T(t)\uu=e^{\lambda t}\uu$ for every $t>0$. Since $T(t)$ maps $L^p(\Rd;\Cm)$ into $D(A_p)$, hence in $W^{2,p}(\Rd;\Cm)$, from the Sobolev embedding theorems, the consistency of the semigroup $(T(t))$ and (if needed) a bootstrap argument, we can infer that $T(t)\uu$ belongs to $L^q(\Rd;\Cm)$. Hence, $e^{\lambda t}\uu$, or equivalently $\uu$, belongs to $L^q(\Rd;\Cm)$. Consequently, $T(t)\uu\in D(A_q)$, so that $\uu$ itself belongs to this space and, by the consistence of $A_r$ in the $L^r$ scale, $\lambda\uu=A_q\uu$, so that $\lambda\in P\sigma(A_q)$.

Let us now prove the second set inclusion in the statement of the proposition. We split the proof into two steps.

{\em Step 1}. Here, we prove that, if $\lambda\in\mathbb{C}$ and $(\uu_n)\subset D(A_s)$, for some $s\in (1,\infty)$, is a sequence such that $\limsup\limits_{n\to \infty}\|\uu_n\|_{L^s(B(n);\mathbb{C})}>0$ and for every $n \in \N$, $\|\uu_n\|_{L^s(\Rd;\Cm)}=1$  and  $\|\f_n\|_{L^s(\Rd;\Cm)}\le n^{-1}$ where $\f_n:=\lambda\uu_n-A_s\uu_n$, then $\lambda\in A\sigma(A_r)$ for every $r\in\left (1\vee\frac{ds}{d+s},s\right ]$.

So, let us fix $\lambda$ and $(\uu_n)$ as above. Note that, since $A_s\uu_n=\lambda\uu_n-\f_n$, we can estimate
\begin{align*}
\|A_s\uu_n\|_{L^s(\Rd;\Cm)}\le |\lambda|\|\uu_n\|_{L^s(\Rd;\Cm)}
+\|\f_n\|_{L^s(\Rd;\Cm)}
\le |\lambda|+1,
\end{align*}
so that the sequence $(\uu_n)$ is bounded in $D(A_s)$. From \eqref{norm_eq}, we conclude that the sequence $(\uu_n)$ is bounded in $W^{2,s}(\Rd;\Cm)$.

Next, we introduce a smooth function $\vartheta$, such that $\vartheta\equiv 1$ in $B(1)$ and $\vartheta\equiv 0$ outside the ball $B(2)$, and, for every $n\in\N$, we consider the function $\vartheta_n$, defined by $\vartheta_n(x)=\vartheta(n^{-1}x)$ for every $x\in\R^d$. Let us set $\vv_n=\vartheta_n\uu_n$ for every $n\in\N$. Then, $\vv_n$ belongs to $D(A_r)$ for every $r\in (1,s]$ and
\begin{align*}
\lambda\vv_n-A_r\vv_n=\f_n\vartheta_n-\sum_{h,k=1}^d(Q^{hk}+Q^{kh})D_h\uu_n D_k\vartheta_n-\sum_{h,k=1}^dQ^{hk}\uu_n D_{hk}\vartheta_n=:\bm g_n.
\end{align*}

Observe that
\begin{align*}
\|(Q^{hk}+Q^{kh})D_h\uu_n D_k\vartheta_n\|_{L^r(\Rd;\Cm)}^r
\le & \|Q^{hk}\!+\!Q^{kh}\|_{\infty}^r\int_{\Rd}|D_h\uu_n|^r|D_k\vartheta_n|^rdx\\
\le &\|Q^{hk}\!+\!Q^{kh}\|_{\infty}^r
\bigg (\int_{\Rd}|D_h\uu_n|^sdx\bigg )^{\frac{r}{s}}\!
\bigg (\int_{\Rd}|D_k\vartheta_n|^{\frac{sr}{s-r}}dx\bigg )^{1-\frac{r}{s}}\\
\le & \|Q^{hk}\!+\!Q^{kh}\|_{\infty}^r\|D_h\uu_n\|_{L^s(\Rd;\Cm)}^r\|D_k\vartheta\|_{\infty}^rn^{-r}|B(2n)|^{1-\frac{r}{s}}\\
=& Cn^{d\left (1-\frac{r}{s}\right )-r}\sup_{\ell\in \N}\|\uu_\ell\|_{W^{2,s}(\Rd;\Cm)}^r
\end{align*}
for every $n\in\N$, every $h,k=1,\ldots,d$, every $r\in (1,s)$ and some positive constant $C$, where $|B(2n)|$ denotes the Lebesgue measure of the ball $B(2n)$. Similarly,
\begin{align*}
\|Q^{hk}\uu_n D_{hk}\vartheta_n\|_{L^r(\Rd;\Cm)}\le
C'n^{d\left (\frac{1}{r}-\frac{1}{s}\right )-2}\sup_{\ell\in\N}\|\uu_{\ell}\|_{W^{2,s}(\Rd;\Cm)}
\end{align*}
and
\begin{eqnarray*}
\|\f_n\vartheta_n\|_{L^r(\Rd;\Cm)}\le C''\|\f_n\|_{L^s(\Rd;\Cm)}n^{d\left (\frac{1}{r}-\frac{1}{s}\right )}\le
C''n^{d\left (\frac{1}{r}-\frac{1}{s}\right )-1}
\end{eqnarray*}
for every $n$, $h$ and $k$ as above and some positive constants $C'$ and $C''$.

Summing up, we have proved that
\begin{align*}
\|\bm g_n\|_{L^r(\Rd;\Cm)}\le C_*n^{d\left (\frac{1}{r}-\frac{1}{s}\right )-1}
\end{align*}
for every $n\in\N$, every $r\in (1,s]$ and some positive constant $C_*$, independent of $n\in\N$. Hence,
if $r\in \left (\frac{ds}{d+s},s\right ]$, then the sequence $(\bm g_n)$ converges to zero in $L^r(\Rd;\Cm)$.

To conclude that $\lambda\in A\sigma(A_r)$ for $r$ as above, we need to show that the sequence $\vv_n$ does not vanish in $L^r(\Rd;\Cm)$ as $n$ tends to $\infty$. Suppose, by contradiction, that $\|\vv_n\|_{L^r(\Rd;\Cm)}$ vanishes as $n$ tends to $\infty$. Since $A_r\vv_n=\lambda\vv_n-\g_n$ for every $n\in\N$ and $\g_n$ converges to $\bm{0}$ in $L^r(\Rd;\Cm)$, the sequence $(A_r\vv_n)$
vanishes in $L^r(\Rd;\Cm)$ as $n$ tends to $\infty$. Hence, $\vv_n$ vanishes in $D(A_r)$ and, taking \eqref{norm_eq} into account, we deduce that $\vv_n$ vanishes in $W^{2,r}(\Rd;\Cm)$ as $n$ tends to $\infty$. Now, we apply the Sobolev embedding theorems to get to a contradiction. We distinguish two cases.
\begin{itemize}
\item
If $r\ge \frac{d}{2}$, then
$W^{2,r}(\Rd;\Cm)$ is continuously embedded into $L^t(\Rd;\Cm)$ for every $t\in [r,\infty)$. In particular, $W^{2,r}(\Rd;\Cm)$ is continuously embedded into $L^s(\Rd;\Cm)$.
\item
If $r<\frac{d}{2}$, then $W^{2,r}(\Rd;\Cm)$ is continuously embedded into $L^t(\Rd;\Cm)$ for every $t\in \left [r,\frac{dr}{d-2r}\right ]$. As it is immediately seen, the right-hand side of the previous interval il larger than $s$. Hence, also in this case $W^{2,r}(\Rd;\Cm)$ is continuously embedded into $L^s(\Rd;\Cm)$.
\end{itemize}

From the above arguments, we conclude that there exists a positive constant $c$ such that
$\|\vv_n\|_{L^s(\Rd;\Cm)}\le c\|\vv_n\|_{W^{2,r}(\Rd;\Cm)}$ for every $n\in\N$, so that $\|\vv_n\|_{L^s(\Rd;\Cm)}$ vanishes as $n$ tends to $\infty$. This is a contradiction, since $\|\vv_n\|_{L^s(\Rd;\Cm)}\ge\|\uu_n\|_{L^s(B(n);\Cm)}$ and this latter norm does not vanish as $n$ tends to $\infty$, by assumptions.

Based on the above results, we can determine  a strictly increasing sequence $(k_n)\subset\N$ such that $\|\vv_{k_n}\|_{L^r(\Rd;\Cm)}\ge c'$
for every $n\in\N$ and some positive constant $c'$. Hence, the sequence
$(\ww_n)$, defined by $\ww_n=\|\vv_{k_n}\|_{L^r(\Rd;\Cm)}^{-1}\vv_{k_n}$ for every $n\in\N$, belongs to the boundary of the unit ball of $L^r(\Rd;\Cm)$ and $\lambda\ww_n-A_r\ww_n$ vanishes in $L^r(\Rd;\Cm)$ as $n$ tends to $\infty$. Hence, $\lambda\in A\sigma(A_r)$ for every $r\in\left (1 \vee \frac{ds}{d+s},s\right ]$.

{\em Step 2}. Here, we assume that $\lambda\in P\sigma(A_q)$ and prove that it belongs to the approximate spectrum of $A_p$. We denote by $\bm{0}\not\equiv \uu\in D(A_q)$ an eigenfunction of $A_q$ associated with the eigenvalue $\lambda$, with $\|\uu\|_{L^q(\Rd;\Cm)}=1$.
Since $P\sigma(A_q)\subseteq A\sigma(A_q)$ and $\|\uu\|_{L^q(B(n);\Cm)}$ converges to $\|\uu\|_{L^q(\Rd;\Cm)}>0$ as $n$ tends to $\infty$, the assumptions of Step 1 are satisfied with $\uu_n=\uu$ for every $n\in \N$. Thus, it follows that $\lambda\in A\sigma(A_r)$ for every
$r\in\left (1\vee\frac{qd}{q+d},q\right ]$. If $\frac{qd}{q+d}<1$, then we get the claim. On the other hand, if $\frac{qd}{q+d}\ge 1$ and $\frac{qd}{q+d}<p$, then $\lambda\in A\sigma(A_p)$ and we are done. On the contrary, if $\frac{qd}{q+d}\ge p$, then we use a bootstrap argument to show that $\lambda\in A\sigma(A_p)$.
For this purpose, we observe that
$\lambda\in A\sigma(A_{r_2})$, where
$r_2=\frac{1}{2}\left (q+\frac{dq}{d+q}\right )$. Moreover, still from Step 1, we can infer that there exists a strictly increasing sequence $(k_{n,1})\subset\N$ such that the sequence $(\vv_{n,1})\in D(A_{r_2})$, defined by
$\vv_{n,1}=\uu\vartheta_{k_{n,1}}$ for every $n\in\N$ is such that $\|\f_{n,1}\|_{L^{r_2}(\Rd;\Cm)}\le n^{-1}$ for every $n\in\N$, where
$\f_{n,1}:=\lambda\vv_{n,1}-A_{r_2}\vv_{n,1}$ vanishes in $L^r(\Rd;\Cm)$ as $n$ tends to $\infty$. Note that $\|\vv_{n,1}\|_{L^{r_2}(B(n);\Cm)}
=\|\uu\|_{L^{r_2}(B(n);\Cm)}$ for every $n\in\N$. The monotone convergence theorem shows that
$\|\uu\|_{L^{r_2}(B(n);\Cm)}$ tends to $\|\uu\|_{L^{r_2}(\Rd;\Cm)}$ if $\uu\in L^{r_2}(\Rd;\Cm)$ (and this latter norm is positive since $\uu$ does not vanish almost everywhere in $\Rd$), and $\|\uu\|_{L^{r_2}(B(n);\Cm)}$ tends to $\infty$
if $\uu\notin L^{r_2}(\Rd;\Cm)$.
We have so proved that all the assumptions of Step 1 are satisfied, so that $\lambda\in A\sigma(A_r)$ for every $r\in \left (1 \vee\frac{dr_2}{d+r_2},r_2\right ]$.  Arguing as above, if $\frac{dr_2}{d+r_2}<1$ or if $1\le \frac{dr_2}{d+r_2}<p$, then we are done. Otherwise, we iterate the argument. We set $r_3=\frac{1}{2}\left (r_2+\frac{dr_2}{d+r_2}\right )$
and consider a sequence $(k_{n,2})\subset (k_{n,1})$ such that the sequence $(\vv_{n,2})$, defined by $\vv_{n,2}=\uu\vartheta_{k_{n,2}}$ for every $n\in\N$, satisfies the conditions $\lim_{n\to\infty}\|\vv_{n,2}\|_{L^{r_3}(B(n);\Cm)}\in (0,\infty)\cup\{\infty\}$ and $\|\lambda\vv_{n,2}-A_{r_3}\vv_{n,2}\|_{L^{r_3}(B(n);\Cm)}\le n^{-1}$ for every $n\in\N$.
Therefore, we can apply Step 1 once more and infer that
$\lambda\in A\sigma(A_r)$ for every $r\in \left (1 \vee \frac{r_3d}{d+r_3},r_3\right ]$. If $p$ belongs to such an interval, then we are done, otherwise we apply further the same argument and in a finite number of steps we conclude that $\lambda\in A\sigma(A_p)$. Indeed, the sequence $(r_n)$, defined recursively by
\begin{eqnarray*}
\left\{
\begin{array}{ll}
r_{n+1}=\displaystyle\frac{1}{2}\bigg (r_n+\frac{dr_n}{r_n+d}\bigg ), & n\in\N,\\[3mm]
r_1=q,
\end{array}
\right.
\end{eqnarray*}
is decreasing and converges to $0$ as $n$ tends to $\infty$.
\end{proof}

As already pointed out in the introduction, in general, one cannot expect that the spectrum of the realization of an elliptic operator in $L^p(\Rd;\Cm)$ is independent of $p\in (1,\infty)$ even in the scalar case.
Indeed, there are many situations where the spectrum does depend on $p$, see for instance \cite{arendt,davies, dav_sim_tay, Fin_isi, sturm}.

In the following theorem, we provide a sufficient condition for the equality $\sigma(A_p)=\sigma(A_q)$ for $p$ and $q$ varying in $(1,\infty)$.

\begin{thm}
\label{thm-comp}
Assume that there exists a function $\psi:\Rd\to\R$, diverging to $\infty$ as $|x|$ tends to $\infty$, such that $|V(x)\eta|\ge\psi(x)|\eta|$ for every $x\in\Rd$ and $\eta\in\Rm$. Then, the spectrum of the operator $A_p$ is independent of $p\in (1,\infty)$ and consists of eigenvalues only. Moreover, for every $t>0$ the operator $T(t)$ is compact.
\end{thm}

\begin{proof}
We fix $p\in (1,\infty)$ and, to begin with, we show that $D(A_p)$ is compactly embedded into $L^p(\Rd;\Cm)$. For this purpose, we observe that, if $\uu$ belongs to the closed unit ball of $D(A_p)$, then $V\uu$ belongs to $L^p(\Rd;\Cm)$ and $\|V\uu\|_{L^p(\Rd;\Cm)}\le 1$. Since $|V\uu|\ge \psi |\uu|$, for every $\varepsilon>0$ we can determine $r>0$ such that
\begin{equation}
\int_{\Rd\setminus B(r)}|\uu(x)|^pdx\le \varepsilon
\label{elisabetta}
\end{equation}
for every $\uu$ in the unit ball of $D(A_p)$. Recalling that $D(A_p)\subset W^{2,p}(\Rd;\Cm)$ and $W^{2,p}(B(r);\Cm)$ is compactly embedded into $L^p(B(r);\Cm)$, from
\eqref{elisabetta} it follows that the unit ball of $D(A_p)$
is totally bounded in $L^p(\Rd;\Cm)$. As a consequence, $D(A_p)$ compactly embeds into $L^p(\Rd;\Cm)$. From this result, we deduce that the operator $R(\lambda,A_p)$ is compact for every $\lambda\in\rho(A_p)$. Since the space of compact operators is a closed subspace of the space of all bounded operators from $L^p(\Rd;\Cm)$ into itself and, for every $t>0$, the operator $T(t)$ is defined via the Dunford integral, we conclude that each operator $T(t)$ is compact as well.

Further, we observe that the compactness of the resolvent operators $R(\lambda,A_p)$ implies that $\sigma(A_p)$ consists of eigenvalues only.

To complete the proof, we show that $\sigma(A_p)$ is independent of $p\in (1,\infty)$. For this purpose, we fix $p,q\in (1,\infty)$, with $p<q$. By Proposition \ref{prop-4.3}, we know that $P\sigma(A_p)\subseteq P\sigma(A_q)$
and $P\sigma(A_q)\subseteq A\sigma(A_p)$. Since the spectrum of each operator $A_p$ consists of eigenvalues only, the equality $\sigma(A_p)=\sigma(A_q)$ follows immediately.
\end{proof}

Theorem \ref{thm-comp} shows that, if the semigroup $(T(t))$ is compact on the $L^p$-scale, then its spectrum is independent of $p$. In the following theorem, we consider the operator $\mathcal L$, in divergence form, i.e., we assume that
\begin{eqnarray*}
\bm{\mathcal L}\uu=\sum_{h,k=1}^dD_h(Q^{hk}D_k\uu)-V\uu,
\end{eqnarray*}
on sufficiently smooth functions $\uu$, and we prove some inclusion of the spectra of its realization in $L^p(\Rd;\Cm)$ when $p$ varies in suitable subsets of $(1,\infty)$.

\begin{thm}
Suppose that the diffusion coefficients of the operator $\mathcal L$ belong to $W^{1,\infty}(\Rd)$  and that $V$ is symmetric.
Then, for every $p\in (1,\infty)$, the realization $L_p$ of the operator $\mathcal L$ in $L^p(\Rd;\Cm)$, with domain $D(L_p)=
\{\uu\in W^{2,p}(\Rd;\Cm): V\uu\in L^p(\Rd;\Cm)\}$, generates a strongly continuous analytic semigroup. Moreover,
$\sigma(L_q)\subseteq\sigma(L_p)$ for any $q$ belonging to the interval with endpoints $p$ and $p'$.
\end{thm}

\begin{proof}
Since the diffusion coefficients belong to $W^{1,\infty}(\Rd)$, we can rewrite the operator $\mathcal L$ in the form
\begin{eqnarray*}
\mathcal L\uu=\sum_{h,k=1}^dQ^{hk}D_{hk}\uu+\sum_{h,k=1}^dD_hQ^{hk}D_k\uu-V\uu=\A\uu+\sum_{h,k=1}^dD_hQ^{hk}D_k\uu
\end{eqnarray*}
on smooth enough functions $\uu$. Note that
\begin{align*}
\bigg\|\sum_{h,k=1}^dD_hQ^{hk}D_k\uu\bigg\|_{L^p(\Rd;\Cm)}
\le &\sum_{h,k=1}^d\|D_hQ^{hk}\|_{\infty}\|D_k\uu\|_{L^p(\Rd;\Cm)}\\
\le &C\sum_{k=1}^d\|D_k\uu\|_{L^p(\Rd;\Cm)}\\
\le &C'\|\uu\|_{L^p(\Rd;\Cm)}^{\frac{1}{2}}
\|\uu\|_{W^{2,p}(\Rd;\Cm)}^{\frac{1}{2}}\\
\le &C''\|\uu\|_{L^p(\Rd;\Cm)}^{\frac{1}{2}}
\|\uu\|_{D(A_p)}^{\frac{1}{2}}\\
\le &\varepsilon \|\A\uu\|_{L^p(\Rd;\Cm)}
+C_{\varepsilon}\|\uu\|_{L^p(\Rd;\Cm)}
\end{align*}
for every $\varepsilon>0$ and some positive constants $C$, $C'$, $ C''$ and $C_{\varepsilon}$, this latter one blowing up as $\varepsilon$ tends to $0$.

By Theorem \ref{main}, the realization $A_{p}$ of operator $\A$, with domain $D(A_p)=D(L_p)$, generates a strongly continuous analytic semigroup in $L^p(\Rd;\Cm)$. Moreover, the above computations show that the operator $\sum_{h,k=1}^dD_hQ^{hk}D_k$ is a $A_{p}$-bounded perturbation with $A_{p}$-bound equal to zero. From these two properties it follows  that the operator $L_p$ generates a strongly continuous  analytic semigroup (see \cite[III, Theorem 2.10]{engnagel}).

To prove the last statement it is enough to consider  $p\in (1,2)$ and $q\in (p,p')$ and show that $\rho(L_p)\subseteq \rho(L_q)$. For this purpose, let us fix  $\lambda \in \rho(L_p)$ ($=\rho(L_{p'})$) (see e.g., \cite[Appendix B, Corollary B.12]{engnagel}) and observe that the consistency of the semigroup $(T(t))$ on the $L^r$-scale ($r\in (1,\infty)$), proved in
Proposition \ref{prop-4.2}, shows that $R(\lambda,L_p)$
and $R(\lambda,L_{p'})$, which are bounded operators in $L^{p}(\Rd;\Cm)$ and in $L^{p'}(\Rd;\Cm)$, respectively, coincide on $L^p(\Rd;\Cm)\cap L^{p'}(\Rd;\Cm)$. By the Riesz-Thorin interpolation theorem, we can infer that
$R(\lambda,L_p)$ can be extended with a bounded operator $R_{\lambda}$ in $L^q(\Rd;\Cm)$. We claim that, for every $\f\in L^q(\Rd;\Cm)$, the function $R_{\lambda}\f$ is the unique solution to the equation $\lambda\uu-L_q\uu=\f$ which belongs to $D(A_q)$. As far as the uniqueness of the solution is concerned, we
use the same bootstrap argument as in the proof of Proposition \ref{prop-4.3}  to infer that, if $\lambda$ is an eigenvalue of $L_q$, i.e., if the equation $\lambda\uu-L_q\uu=\bm{0}$ admits a nontrivial solution in $D(A_q)$, then $\lambda\in\sigma(L_p)$, which is a contradiction. Showing that $R_{\lambda}\f$ is a solution to the previous equation in $D(A_q)$ is easy to prove if $\f\in C^{\infty}_c(\Rd;\Cm)$. Indeed, in this case $\uu=R_{\lambda}\f=R(\lambda,L_p)\f=R(\lambda,L_{p'})\f$ and, therefore, $\uu\in W^{2,p}(\Rd;\Cm)\cap W^{2,p'}(\Rd;\Cm)$ and $V\uu\in L^p(\Rd;\Cm)\cap L^{p'}(\Rd;\Cm)$. H\"older inequality shows that $\uu\in W^{2,q}(\Rd;\Cm)$ and $V\uu\in L^q(\Rd;\Cm)$. Therefore, $\uu\in D(A_q)$ and, clearly, it solves the equation $\lambda\uu-A_q\uu=\f$.

For a general $\f\in L^q(\Rd;\Cm)$, we consider a sequence $(\f_n)\subset C^{\infty}_c(\Rd;\Cm)$ converging to $\f$ in $L^q(\Rd;\Cm)$.
Then, $\uu_n=R_{\lambda}\f_n$ converges in $L^q(\Rd;\Cm)$ to the function $\uu=R_{\lambda}\f$. By difference, $L_q\uu_n=\lambda\uu_n-\f_n$ converges in $L^q(\Rd;\Cm)$. Since $L_q$ is a closed operator, it follows that $\uu\in D(A_q)$ and $A_q\uu=\lambda\uu-\f$. We have so proved that $\lambda\in\rho(L_q)$ and the inclusion $\rho(L_p)\subset\rho(L_q)$ follows.
\end{proof}

\section{Examples}
Here, we exhibit a class of operators defined as in \eqref{operator} to which our results can be applied.
\subsection{On the diffusion matrices $Q^{hk}$}
Concerning the diffusion matrices $Q^{hk}$, we recall that they have to satisfy Hypothesis \ref{H1}. Thus, we can consider, for instance, the following situation.

For any $h,k=1, \ldots, d$, the diffusion matrices $Q^{hk}$ can  have the form
\begin{equation*}
Q^{hk}= q_{hk}Q_0+A^{hk}
\end{equation*}
where both $q_{hk}:\Rd\to \R$ and the entries of the $m\times m$ matrix-valued functions $A^{hk}$ belong to $BUC(\Rd;\R)$, and  $Q_0$ is a $m\times m$ real valued matrix that satisfies the condition $(Q_0 \eta, \eta)\ge q_0|\eta|^2$ for any $\eta \in \R^m$ and some positive constant $c_0$.
Setting $Q=(q_{hk})_{h,k=1}^d$, we  assume that there exists a positive function $\lambda_Q:\Rd\to \R$ such that $(Q(x)\xi,\xi)\ge \lambda_Q(x)|\xi|^2$ for any $x,\xi\in \Rd$ and
\begin{eqnarray*}
\inf_{x \in \Rd}\left(q_0\lambda_Q(x)-dm\max_{i,j=1,\ldots, m}\max_{h,k=1, \ldots,d}|A^{hk}_{ij}(x)|\right)= :c_1>0.
\end{eqnarray*}

First, we note that the Legendre-Hadamard ellipticity condition in Hypothesis \ref{H1} is equivalent to the existence of a positive  constant $C$ such that
\begin{equation*}
\sum_{h,k=1}^d\sum_{i,j=1}^mq^{hk}_{ij}\eta_i\eta_j\xi_h\xi_k\ge C|\eta|^2|\xi|^2,\qquad\;\,\xi\in\Rd,\;\,\eta\in\R^m.
\end{equation*}

 In this case it holds
\begin{align*}
\sum_{i,j=1}^m\sum_{h,k=1}^d q^{hk}_{ij}(x)\eta_i\eta_j\xi_k\xi_h=& (Q_0\eta, \eta)(Q(x)\xi,\xi)+\sum_{i,j=1}^m\sum_{h,k=1}^d A^{hk}_{ij}(x)\eta_i\eta_j\xi_k\xi_h\\
\ge & q_0\lambda_Q(x)|\eta|^2|\xi|^2-\max_{i,j=1,\ldots, m}\max_{h,k=1, \ldots,d}|A^{hk}_{ij}(x)|\sum_{i,j=1}^m\sum_{h,k=1}^d|\xi_h||\xi_k||\eta_i||\eta_j|\\
\ge & q_0\lambda_Q(x)|\eta|^2|\xi|^2-dm\max_{i,j=1,\ldots, m}\max_{h,k=1, \ldots,d}|A^{hk}_{ij}(x)||\xi|^2|\eta|^2\\
\ge &c_1|\xi|^2|\eta|^2
\end{align*}
for any $\xi\in \Rd$ and $\eta\in \Rm$, hence condition \eqref{son} is satisfied with $C=c_1$.

\subsection{On the potential matrix $V$: the case when Hypotheses \ref{H1-1} are satisfied}\label{funcf}
\begin{example}\label{ex_1} {\rm
Let  $f$ be a function belonging to $W^{2,\infty}_{{\rm loc}}(\Rd)$, with positive infimum  and such that, for some $\alpha\in(0,1)$,  $f^{-\alpha}D_kf$ and $f^{-\alpha}D_{hk}f$ belong to $L^\infty(\Rd)$ for every $h,k\in\{1, \dots, d\}$.
For example, one can consider $f(x)=p(|x|)$ for every $x\in\Rd$, where $p:\R\rightarrow \R$ is a polynomial with positive infimum without first grade term, or $f(x)=(1+|x|^2)^r$ for every $x\in\Rd$, with $r>0$, or, again, $f(x)=|x|^2\log|x|$ outside a ball centered in the origin.

Next, consider the matrix valued-function $V:\Rd\to\R^{m^2}$, defined by
\begin{equation*}
V(x)= f(x)V_0, \qquad\;\, x \in \Rd,
\end{equation*}
where   $V_0$ is a $m\times m$ real-matrix which satisfies one of the following two conditions:
\begin{enumerate}[\rm (a)]
\item
$V_0$ is a symmetric and positive definite matrix;
\item
$V_0$ is a diagonal perturbation of a skewsymmetric real matrix,i.e.
\begin{eqnarray*}
V_0={\rm diag}(v_{11},\ldots, v_{mm})+\tilde{V_0},
\end{eqnarray*}
with $v_{ii}>0$ for any $i=1, \ldots,m$  and $\tilde{V_0}$ is a skewsymmetric matrix with entries $v_{ij}$ with $i,j=1, \ldots, m$, $i\neq j$.
\end{enumerate}

If condition (a) is satisfied  then estimate \eqref{hyp-c0} holds true with $c_0=\lambda_0\inf_{x\in\Rd}f(x)$, where $\lambda_0$ denotes the minimum eigenvalue of the matrix $V_0$. Further, condition \eqref{ana_V} is satisfied as well, since ${\rm Im}( V(x)\zeta,\zeta)=0$ for any $x\in \Rd$ and $\zeta \in \Cm$. Analogously, if we assume that condition (b) is true, then \eqref{hyp-c0} is satisfied with $c_0= \left (\min_{i=1, \ldots,m}v_{ii}\right )\inf_{x\in\Rd}f(x)$. Concerning estimate \eqref{ana_V}, we  observe that, for any $x \in \Rd$ and any $\eta=(\eta_1, \ldots, \eta_m)\in \Cm$, it follows that
\begin{align*}
|{\rm Im}(V(x)\eta,\eta)|
&\le f(x)\bigg |\sum_{i=1}^m\sum_{j\in \{1,\ldots, m\}\setminus{i}}v_{ij}\eta_j\overline{\eta}_i\bigg |\\
& \le \frac{1}{2}f(x)\sum_{i=1}^m\sum_{j\in \{1,\ldots, m\}\setminus{i}}|v_{ij}|(|\eta_i|^2+|\eta_j|^2)\\
& \le f(x)\sum_{i=1}^m|\eta_i|^2\sum_{j\in \{1,\ldots, m\}\setminus{i}}|v_{ij}|\\
& \le c_3 f(x)\sum_{i=1}^mv_{ii}|\eta_i|^2\\
&= c_3 {\rm Re}(V(x)\eta,\eta),
\end{align*}
where $c_3$ is a positive constant such that $\sum_{j\in \{1,\ldots, m\}\setminus{i}}|v_{ij}|\le c_3 v_{ii}$ for every $i=1,\ldots,m$. Estimate \eqref{ana_V} is satisfied with $\omega_0=\arctan(c_3)$.

Clearly, the matrix-valued functions $D_kV V^{-\alpha}$ and $D_{hk}V V^{-\alpha}$ are bounded in $\Rd$ for every $h,k=1,\ldots,d$. Hence, the second part of Hypotheses \ref{H1-1}(ii) holds true.

Finally, concerning the first part of Hypotheses \ref{H1-1}(ii), i.e., the condition $Q^{hk}V=VQ^{hk}$, is satisfied if
$V_0$ commute with the matrices $Q_0$ and $A^{hk}$ for any $h,k=1,\ldots,d$.

\begin{rmk}
{\rm We point out, that in the weakly coupled case (i.e., when $Q^{hk}=q_{hk}I$ for every $h,k=1,\ldots,d$),  the  condition $Q^{hk}V=VQ^{hk}$ is trivially satisfied for every $h,k=1,\ldots,d$. and the potential matrix $V$ has entries that can grow more than quadratically, which was the best case considered in \cite{KLMR}.}
\end{rmk}}
\end{example}

Now, we provide examples of potential matrices $V$ that can have a more general form according to Hypotheses \ref{hyp_2}.
\subsection{On the potential matrix V: the case when Hypotheses \ref{hyp_2} are satisfied}

\begin{example}\label{ex_2}
{\rm

Fix $r>0$ and assume that
\begin{equation}\label{ch_1}
v_{ij}(x)=g_{ij}(x)(1+|x|^2)^{\beta_{ij}},\qquad\;\, x \in \Rd,\;\, i\neq j,
\end{equation}
and
\begin{equation}\label{ch_2}
v_{ii}(x)=(1+|x|^2)^r+g_i(x)(1+|x|^2)^{\beta_i},\qquad\,\, x \in \Rd,\;\, i=1,\ldots,m,
\end{equation}
with $g_i, g_{ij} \in
C_b(\Rd)$ and  $\beta_i, \beta_{ij}\in(0,r)$ and $v(x)=(1+|x|^2)^r$ for every $x\in\Rd$.
Then Hypotheses \ref{hyp_2} are  satisfied with $\alpha \in (1-\frac{1}{2r},1)$ and
$\gamma=r^{-1}\beta$, where $\beta=\max\{\beta_i, \beta_{i,j}: i,j=1,\ldots,m,\ j\neq i\}$.}
\end{example}

Now, let us observe that both in Example \ref{ex_1}(a), if the function $f$ diverges to $\infty$ as $|x|$ tends to $\infty$, and in Example \ref{ex_2}, there exists a function $\psi$ diverging to $\infty$ as $|x|$ tends to $\infty$ such that

\begin{equation}\label{comp_cond}
|V(x)\eta|\ge\psi(x)|\eta|,\qquad\;\, x \in \Rd,\;\, \eta \in \Rm.
\end{equation}
Thus, Theorem \ref{thm-comp}
can be applied to deduce that:\\
(i) the spectrum of the operator $A_p$ is independent of $p\in (1,\infty)$ and consists of eigenvalues only,\\
(ii) the operator $T(t)$ is  compact in $L^p(\Rd;\Rm)$ for any $t>0$.

Indeed, in Example \ref{ex_1}, if the function $f$ diverges to $\infty$ as $|x|$ tends to $\infty$, it can be easily shown that
condition \eqref{comp_cond} is satisfied with $\psi(x)=\lambda_0 f(x)$ and, clearly, $\psi(x)$ diverges to $\infty$ as $|x|$ tends to $\infty$.
On the other hand, in Example \ref{ex_2}, with the choices in \eqref{ch_1} and \eqref{ch_2},  we can estimate
\begin{align}
|V(x)\eta|^2
=& \sum_{i=1}^m v_{ii}^2\eta_i^2+\sum_{i=1}^m\bigg(\sum_{j\in\{1,\ldots,m\}\setminus\{i\}} v_{ij}\eta_j\bigg)^2
+2\sum_{i=1}^mv_{ii}\eta_i\sum_{j\in\{1,\ldots,m\}\setminus\{i\}}v_{ij}\eta_j\notag\\
\ge &
\sum_{i=1}^m v_{ii}^2\eta_i^2
-2\sum_{i=1}^m\sum_{j\in\{1,\ldots,m\}\setminus\{i\}}|v_{ii}||v_{ij}||\eta_i||\eta_j|\notag\\
\ge & \sum_{i=1}^m v_{ii}^2\eta_i^2-
\sum_{i=1}^m\sum_{j\neq i}|v_{ii}||v_{ij}|(\eta_i^2+\eta_j^2).
\label{starrrr}
\end{align}
for any $\eta\in\Rm$.

Now, if $\mathfrak{g}=\max_{i,j=1,\ldots,m, i\neq j}\{\|g_i\|_{\infty},\|g_{ij}\|_{\infty}\} $, then we can estimate
\begin{equation*}
|v_{ii}(x)|\le (1+|x|^2)^r\big (1+\|g_i\|_{\infty}(1+|x|^2)^{\beta_i-r})
\le (1+\mathfrak{g})(1+|x|^2)^r
\end{equation*}
and $|v_{ij}(x)|\le \mathfrak{g}(1+|x|^2)^\beta$ for every $x \in \Rd$. Thus, $|v_{ij}(x)||v_{ii}(x)| \le \mathfrak{g}(1+\mathfrak{g})(1+|x|^2)^{\beta+r} $ for every $x \in \Rd$. Consequently,
\begin{align}\label{pippo1}
\sum_{i=1}^m\sum_{j\neq i}|v_{ii}||v_{ij}|(\eta_i^2+\eta_j^2)&\le  \mathfrak{g}(1+\mathfrak{g})(1+|x|^2)^{\beta+r}\sum_{i=1}^m\sum_{j\neq i}(\eta_i^2+\eta_j^2)\notag \\
&\le 2m\mathfrak{g}(1+\mathfrak{g})(1+|x|^2)^{\beta+r}|\eta|^2
\end{align}
for every $x\in\Rd$ and $\eta\in\Rm$.
Moreover,
\begin{equation}
\label{pippo}
|v_{ii}(x)|
\ge (1+|x|^2)^r\big (1-\mathfrak{g}(1+|x|^2)^{\beta-r})\ge \frac{1}{2}(1+|x|^2)^r,\qquad\;\,x\in\Rd\setminus B(k),\;\,i=1,\ldots,m,
\end{equation}
for a suitable positive radius $k$.
Therefore, using \eqref{starrrr}, \eqref{pippo1} and \eqref{pippo} we conclude that
\begin{align*}
|V(x)\eta|^2\ge (1+|x|^2)^{2r}\bigg [\frac{1}{4}-2m\mathfrak{g}(1+\mathfrak{g})(1+|x|^2)^{\beta-r}\bigg ]|\eta|^2
\ge \frac{1}{8}(1+|x|^2)^{2r}|\eta|^2
\end{align*}
for every $x\in\Rd\setminus B(k)$ and $\eta\in\Rm$, up to replacing $k$ with a larger radius if needed.
Consequently, \eqref{comp_cond} holds true with
$\psi(x)=(\sqrt{8})^{-1}
(1+|x|^2)^{r}\chi_{\Rd\setminus B(k)}(x)$ for every $x\in\Rd$,
which blows up as $|x|$ tends to $\infty$.

\end{document}